\documentclass{gtpart}     % Basic GT/GTM/AGT style
\usepackage{mathrsfs}
\usepackage{amsfonts}
\usepackage{pinlabel}  %%% the recommended graphics+labelling package
\usepackage{graphicx}  %%% the recommended graphics package
\usepackage[all]{xy}
\usepackage{color}
\usepackage[english]{babel}
\usepackage{amscd}

\title{Diffeomorphism types of simply connected $3$-dimensional Mori fibre spaces}

\author{Feng Hao}
\givenname{Feng}
\surname{Hao}
\address{School of Mathematics, Shandong University, Jinan 250100, China}
\email{feng.hao@sdu.edu.cn}
\author{Yang Su}
\givenname{Yang }
\surname{Su}
\address{State Key Laboratory of Mathematical Sciences, Academy of Mathematics and Systems Science, Chinese Academy of Sciences; University of Chinese Academy of Sciences, Beijing 100190, China}
\email{suyang@math.ac.cn}

\author{Jianqiang Yang}
\givenname{Jianqiang}
\surname{Yang}
\address{School of Mathematical Sciences, Guangxi Minzu University, Nanning, 530006,  China.}
\email{yangjq\_math@sina.com}
\title[Mori fibre spaces]{Diffeomorphism types of simply connected $3$-dimensional Mori fibre spaces}
\newtheorem{thm}{Theorem}[section]    % Standard theorem environment
\newtheorem{lem}[thm]{Lemma}          % Lemma environment with numbering
\newtheorem{proposition}{Proposition}
\newtheorem{corollary}{Corollary}
\theoremstyle{definition}
\newtheorem{defn}[thm]{Definition}    % Definition environment with
\newtheorem{rem}{Remark}             % Unnumbered environment for remarks

\newcommand{\p}{\mathbb P}
\newcommand{\z}{\mathbb Z}
\newcommand{\q}{\mathbb Q}

\newcommand{\oo}{\mathcal O}

\begin{document}

\begin{abstract}    % type your abstract below
In this article, we find finitely many numerical invariants to classify the diffeomorphism types of three dimensional simply connected Mori fibre spaces with torsion free homology groups. 
\end{abstract}

\maketitle
\section{Introduction}

Fano varieties, Calabi-Yau varieties and canonically polarized varieties are simple building blocks in the minimal model program. It is always interesting to understand their differential topological type.  The classification of algebraic varieties up to diffeomorphism is considered to be very helpful for the understanding of algebro-geometric structures on these varieties.
A celebrated result of Koll\'ar-Miyaoka-Mori \cite{KMM92} shows that smooth projective Fano varieties of a given dimension form a bounded family. In particular Fano varieties of a given dimension have only finitely many diffeomorphism or even deformation types. An outstanding well-know open question asks whether there are only finitely many diffeomorphism types of smooth Calabi-Yau varieties of a given dimension $\geq 3$ (The dimension $2$ case is known). However, one can not even control the topological type of canonically polarized varieties.

For lower dimensional Fano varieties, we know that the only one-dimensional Fano
manifold is $\mathbb{P}^1$; two dimensional Fano manifolds, i.e., the del Pezzo
surfaces, lay in 10 diffeomorphism families (see e.g. \cite[p.112]{MiPe97}); Smooth projective Fano 3-folds admit 105 diffeomorphism or even deformation types due to Isikovski and Mori-Mukai \cite{Isk77}\cite{Isk78}\cite{MoMu81}. Very little is known about the diffeomorphism types of Fano manifolds in
higher dimensions. 

In this article, we study the diffeomorphism types of three dimensional Mori fibre spaces (see Definition \ref{def:mfs}, MFS for short), which are basic outputs after running the minimal model program. Specifically, we are interested in three dimensional MFSs with two shared key properties, namely, simply connectedness and homological torsion freeness, with Fano threefolds. Roughly speaking, we find \textit{finitely many numerical invariants} (topological or algebraic geometric) such as the topological Euler characteristic $e(X)$, the holomorphic Euler characteristic $\chi(\oo_X)$, and the characteristic number $K_X^3$, to classify the diffeomorphism types of three dimensional simply connected MFSs  $X$  with torsion free homologies. Our results are sharp in the sense that making the assumptions ``simply connected'' and ``torsion free homology'' is not for technical reasons. In fact, one cannot even find finitely many numerical invariants to classify the fundamental groups or the torsion part of the third cohomology group of 3 dimensional MFSs in general.  Moreover, unlike Fano threefolds, there are infinitely many diffeomorphism types of simply connected three dimensional MFSs with torsion free homologies. For example, consider the simply connected smooth ruled varieties that are birational to  $\mathbb{P}^1\times S$, where $S$ are simply connected surfaces.  The best one can do is to find finitely many numerical invariants to classify them. The primary advantage of this approach is its high practicality, as it significantly simplifies the process of determining whether two MFS are diffeomorphic.

Our main theorem is stated for  $3$-dimensional smooth Mori fiber spaces (MFS for short, see Definition \ref{def:mfs}) as follows. The notations and  invariants appeared Theorem \ref{thm:main} will be explained after the statement of the theorem.

\begin{thm}\label{thm:main}
Let $f \colon X \to Y$, $f' \colon X' \to Y'$ be simply connected $3$-dimensional MFSs with torsion free homology groups. 
\begin{enumerate}
\item If $\dim Y \ne \dim Y'$, then $X$ is not diffeomorphic to $X'$, with the only exception when $X$ is diffeomorphic to  $\p^2 \times \p^1$, or  when $X$ is diffeomorphic to the Fano threefold $\mathbb X$ with Betti numbers $b_2(\mathbb X)=2, \ b_3(\mathbb X)=40$. 
\item If $\dim Y=\dim Y'=0$, then $b_2(X)=b_2(X')=1$, and $X$ is oriented diffeomorphic to $X'$ if and only if they have the same degree and Euler characteristic.
\item If $\dim Y=\dim Y'=1$ and $X$ is oriented diffeomorphic to $X'$, then $K(X,Y)=K(X',Y')$.
\begin{enumerate}
\item When $K(X,Y)=9$,  $X$ is diffeomorphic to $\p^2 \times \p^1$ or $\p(\oo \oplus \oo(-1))$. 
\item When $K(X,Y)=8$, $X$ is oriented diffeomorphic to $X'$ if and only if $K_X^3=K_{X'}^3$. 
\item When $1 \le K(X,Y) \le 5$, $X$ is oriented diffeomorphic to $X'$ if and only if $e(X)=e(X')$, $K_{X/Y}^3=\pm K_{X'/Y'}^3$.
\item When $K(X, Y)=6$,  $X$ is oriented diffeomorphic to $X'$ if and only if $d(X,Y)=d(X',Y')$, $e(X)=e(X')$, and 
$$K_{X/Y}^3=K_{X'/Y'}^3 \ \ \mathrm{or}  \ \ K_{X/Y}^3=-K_{X'/Y'}^3+\frac{12(d(X,Y)-1)}{d(X,Y)} K(X,Y).$$
\end{enumerate}

\item If $\dim Y = \dim Y'=2$, when $f$ is a smooth fibration while $f'$ is singular, then $X$ is not diffeomorphic to $X'$.
\begin{enumerate}
\item Assume both $f$ and $f'$ are smooth fibrations, then $X$ is oriented diffeomorphic to $X'$ if and only if they have the same $w_2$-type, $e(X)=e(X')$, and one of the following conditions is fulfilled
\begin{equation}
     \chi(\oo_X) = \chi(\oo_{X'}), \ \ K_X^3=K_{X'}^3; \tag{A}
     \label{eq:A}
\end{equation} 
or
\begin{equation}
    \chi( \oo_X) = \frac{1}{4}e(X')-\chi(\oo_{X'}), \ \ K_X^3=-K_{X'}^3-12e(X'). \tag{A'}
    \label{eq:A'}
\end{equation}

In this case $b_3(X)=0$.

\item Assume the discriminant curves $C_f$ and $C_f'$ are both non-empty. If $X$ is oriented diffeomorphic to $X'$, then they have the same $w_2$-type, $e(X)=e(X')$, $b_3(X)=b_3(X')$,  $[C_f]$ and $[C_{f'}]$ have the same divisibility and type in $H_2(X,\z)$ and $H_2(X',\z)$, respectively,   
and one of the following conditions is fulfilled
\begin{equation}
\chi( \oo_X) = \chi(\oo_{X'}), \ \  K_X^3=K_{X'}^3, \\
\text{$[C_f]$ and $[C_{f'}]$ have the same norm}
\tag{B}
\label{eq:B}
\end{equation}  
or
\begin{equation}
\begin{aligned}
\chi(\oo_X)  = \frac{1}{4}(e(X)-b_3(X))- \chi(\oo_{X'}), \\
K_X^3  = -12e(X) + 18b_3(X)- K_{X'}^3 \\
\text{$[C_f]$ and $[C_{f'}]$ have opposite norm}
\end{aligned}
\tag{B'}
\label{eq:B'}
\end{equation}

Conversely, except for the case $\chi( \oo_X)=1$ and $b_2(X) \le 10$,  these necessary conditions are sufficient for $X$ being oriented diffeomorphic to $X'$. In the exceptional range, these conditions determine the oriented diffeomorphim type up to finitely many possibilities.
\end{enumerate} 

These invariants satisfy the inequality $e(X)+b_3(X) \ge 6\chi(\oo_X)$, the equality holds if and only if $b_2(X)=2$ (i.e., $Y=\p^2$) and $f$ is a smooth fibration. Moreover,  if $b_2(X)$ is even or $e(X)+b_3(X) \ge 12\chi(\oo_X)$, the cases \eqref{eq:A'} and \eqref{eq:B'} cannot occur. 
\end{enumerate}

\end{thm}

We fix some standard notations. Let $ X $ be a MFS, together with the extremal contraction $f \colon X \to Y$. The canonical divisor of $X$ is denoted by $K_X$, and $K_{X/Y}=K_X-f^*K_Y$ is the relative canonical divisor. We have characteristic numbers $K_X^3=-\langle c_1(X)^3, [X] \rangle$ and  $K_{X/Y}^3=- \langle (c_1(X)-f^*c_1(Y))^3, [X] \rangle$, where $c_1(X)$ is the first Chern class of $X$. 

The $i$-th Betti number of $X$ is denoted by $b_i(X)$, the topological Euler characteristic of $X$ is denoted by $e(X)$, and the holomorphic Euler characteristic of the structure sheaf $\oo_X$ of $X$ is denoted by $\chi(\oo_X)$. 

The \emph{divisibility} of a nonzero element $a$ in a free abelian group $A$ is the unique positive integer $d$ for which $a = d a'$ with $a'$ primitive. Let $\cdot \colon A \times A \to \z$ be a unimodular symmetric bilinear  form over $\z$. The type of a primitive element $a$ is defined as follows: $a$ is \emph{characteristic} if $a \cdot b \equiv b \cdot b \pmod 2$ for all $b \in A$; otherwise $a$ is \emph{ordinary}. The type of an arbitrary element $a$ is the type of the corresponding primitive element $a'$. The \emph{norm} of $a$ is $a \cdot a$ (see \cite{Wall62}). 

We give a brief explanation for the notations and invariants involved in the statement of the theorem. For the detailed description, we refer to the later sections. 

In (1), the Fano threefold $\mathbb X$ with Betti numbers $b_2(X)=2$, $b_3(X)=40$ is the space No. 2 in \cite[Table 12.3]{PS99}.  For a more detailed description   see \S \ref{subsec:mixed}.

In (2),  $X$ is a Fano threefold with Picard number $\rho(X)=1$ and $b_2(X)=1$. The \emph{degree} of $X$   is the nonnegative integer defined as $\langle x^3, [X]\rangle$, where $x \in H^2(X,\mathbb{Z})$ is a generator. In this case the statement of the theorem is a consequence of the stronger deformation classification \cite[Table, p.215]{PS99}.

In (3),  $d(X,Y)$ is the divisibility of $f^*(\sigma_Y)$ in $H^2(X;\z)$, where $\sigma_Y$ is the cohomology fundamental class of $Y$. The invariant $K(X,Y)$ is defined as $K(X,Y)=d(X,Y)(K_{X/Y}^3-K_X^3)/6$. It takes value in $\{1, 2, \cdots, 6, 8, 9\}$. This invariant determines the diffeomorphism type of the smooth fiber $F$ (see equation (\ref{eqn:k})) and is a diffeomorphism invariant of $X$ (Proposition \ref{lem:fiber}). It can be shown that the integer $d(X,Y)$ equals to the product of the orders of the torsion subgroup of $H^2(F_s,\z)$ of all singular fibers $F_s$, and equals to $1$ when $K(X,Y) \ne 6$ (Proposition \ref{prop:y1h2}). The space $\p(\mathcal O \oplus \mathcal O (-1))$ is the projective bundle of the vector bundle $\mathcal O \oplus \mathcal O(-1)$ over $\p^1$.
	
In (4), $[C_f] \in H_2(Y,\mathbb{Z})$ denotes the homology class represented by the discriminant curve $C_f$ of $f \colon X \to Y$. Denote by $w_2(X)$ and $w_2(Y)$ the second Stiefel-Whitney classes of $X$ and $Y$, respectively. Since $f^* \colon H^2(Y, \mathbb Z/2) \to H^2(X, \z/2)$ is injective (see equation (\ref{eq:leray1}) and Proposition \ref{prop:h2}), we may regard $w_2(Y)$ as in $H^2(X, \z/2)$. The $w_2$-\emph{type} of an MFS $f \colon X \to Y$ is defined as follows, reflecting the relation between $w_2(X)$ and $w_2(Y)$.
	\begin{enumerate} 
		\item Type $0$: if $w_2(X)=w_2(Y)=0$, i.e., both $X$ and $Y$ are spin; 
		\item  Type $I$: if $w_2(X) = 0$, $w_2(Y) \ne 0$; 
		\item  Type $II$: if $w_2(X) \ne 0$, $w_2(Y)=0$; 
		\item Type $III_0$: if both $X$ and $Y$ are non-spin, and $w_2(X) - w_2(Y)=0$; 
		\item Type $III_1$: if both $X$ and $Y$ are non-spin, and $w_2(X)-w_2(Y) \ne 0$.
	\end{enumerate} 
We summarize the information in the following table. 	
\begin{center}
\begin{tabular}{|c|c|c|}
\hline
$w_2(X)$ $\backslash$ $w_2(Y)$ & zero & non-zero \\
\hline
zero & type $0$ & type $I$ \\
\hline 
non-zero & type $II$  & type $III_0$:  $w_2(X)-w_2(Y) =0$  \\ \cline{3-3}
& & type $III_1$:  $w_2(X)-w_2(Y)  \ne 0$ \\
\hline

\end{tabular}
\end{center} 

\begin{rem}
In Theorem \ref{thm:main}(4), the cases~\eqref{eq:A'} and \eqref{eq:B'} are ruled out when the invariants are restricted in certain range. For example, a direct calculation shows that if $\chi(\oo_X)=1$, or $\chi(\oo_X)=2$ and $\sum_{\mathrm{even}}b_i(X) \ne 20$, then these cases cannot occur.
\end{rem}

We mention a few applications of the main theorem. Examples of simply connected  projective complex surfaces that are homotopy equivalent to a $K3$ surface but not diffeomorphic to each other are given in \cite{BHPV} and \cite{Kod70}. A direct calculation shows that the projective bundles of the tangent bundles of these surfaces have equal invariants in Theorem \ref{thm:main} (4a), hence they are all diffeomorphic. More general is the following corollary.

\begin{corollary}\label{cor:main-cor}
Let $S_1$ and $S_2$ be two simply connected smooth projective surfaces. If $S_1$ is homotopic equivalent to $S_2$, then the projective bundles of the cotangent bundles or tangent bundles of $S_1$ and $S_2$ are diffeomorphic to each other.
\end{corollary}

Theorem \ref{thm:main} (4) provides inequalities satisfied by the numerical invariants. This puts restriction on the Hodge diamond of MFSs under consideration. For example, the following Hodge diamond does not satisfy the inequality $e(X) + b_3(X) > 6\chi(\oo_X)$, hence cannot be realized by  a simply connected $3$-dimensional MFS with torsion free homology. 
\begin{align*}
\begin{array}{ccccccccccccc}
&&&&& 1 &&&&& \\
&&&& 0 && 0 &&&& \\
&&& 1 && 2 && 1 &&& \\
&& 0 && a && a  && 0 &&\\
&&& 1 && 2 && 1 &&& \\
&&&& 0 && 0 &&&& \\
&&&&& 1 &&&&&
\end{array}
\end{align*}

To conclude the introduction, we give a brief description of the idea used in the proof of Theorem \ref{thm:main}. First, a diffeomorphism classification of simply connected $6$-dimensional manifolds with torsion-free homology was established by Wall and Jupp  more than sixty years ago (see Theorem \ref{thm:6mfds}). In their results, the diffeomorphism type of such manifolds is determined by algebraic topological invariants such as the characteristic classes $w_2(M)$ and $p_1(M)$, and  the cubic form 
$$H^2(M, \mathbb{Z}) \times H^2(M, \mathbb{Z}) \times H^2(M, \mathbb{Z}) \to \mathbb Z, \ \ (x, y ,z) \mapsto \langle x \cup y \cup z , [M] \rangle.$$
Unlike (indefinite) quadratic forms, an explicit classification of cubic forms is still unknown. In this sense, a complete classification of simply connected $6$-manifolds with torsion-free homology groups remains open. However, the special geometric structure of Mori fiber spaces, as stated in Theorem \ref{thm:mori}, makes it possible  to express the associated cubic forms explicitly and to classify them by finitely many numerical invariants.   The basic idea of our approach is as follows. First, for a Mori fiber space $f \colon X \to Y$, we determine the topological invariants required for the classification, namely the triple $(H^*(X, \mathbb{Z}), p_1(X), w_2(X))$. This step requires a combination of techniques from algebraic geometry as well as algebraic and differential topology. In the second step we compare these topological invariants and obtain the diffeomorphism classification. This step is mainly elementary algebra in nature.

The proof of the main theorem proceeds by a case-by-case analysis. We outline the organization of the material in subsequent sections below. In \S \ref{sec:pre} we recall the definition and the structure theorem of MFSs, and prove several key facts about them that will be needed in later sections. We also review the classification theorem of $6$-manifolds by Wall and Jupp.

The proof of Theorem \ref{thm:main} is carried out in  \S \ref{sec:2} and \S \ref{sec:1}. In \S \ref{sec:2} we study the case where the base space $Y$ has dimension $2$. The MFSs are divided into two subclasses, according to whether the discriminant curve $C_f$ of $f$ is empty or nonempty. Formula (\ref{eq:leray1}) and Proposition \ref{prop:h2} imply that MFSs in different classes are not diffeomorphic. We then provide the classification for each class in Theorem \ref{thm:dim2smooth} and Theorem \ref{thm:dim2singular}, respectively.

In \S \ref{sec:1} we study the case where the base space $Y$ has dimension $1$. An invariant $K(X,Y)$ is introduced to distinguish the diffeomorphism types of the smooth fiber of $f$: $\p^2 \# r \overline{\p^2}$, $\p^1 \times \p^1$ or $\p^2$ (see formula (\ref{eqn:k})). For each case,  a corresponding classification theorem is provided (Theorem \ref{thm:k1}, \ref{thm:k2}, \ref{thm:k3}). Finally we study the situation in which a single diffeomorphism type can admit MFS structures with bases of different dimensions. This is addresses in Theorem \ref{thm:mixed}.

Some auxiliary results are established in \S \ref{sec:a}.

\noindent\textit{Acknowledgements.}  We would like to thank Chen Jiang for helpful discussions. The first author is supported by grants NSFC No.1240010723, SDNSFC (No. 2024HWYQ-009, No. ZR2024MA007, No. tsqn202312060). The second author is partially supported by NSFC 12471069. The third author would like to thank the Morningside Center for Mathematics for a research visit in the spring of 2025.

\section{Preliminaries}\label{sec:pre}
In the first part of this section we deduce key algebro-geometric and topological properties of 3-dimensional Mori fiber spaces from their algebraic properties. In the second part we recall the classification of $6$-manifolds by Wall and Jupp. 

We recall the definition of Mori fibre spaces (see e.g. \cite[Definition 3-2-1]{Mat01})

\begin{defn}\label{def:mfs}
A normal projective variety $X$ with only $\mathbb{Q}$-factorial and terminal singularities with a morphism $f\colon X\to Y$ is called a Mori fibre space (MFS for short) if 
\begin{enumerate}
\item $f$ is a morphism with connected fibres onto a normal projective variety $Y$;
\item $\dim Y<\dim X$; and
\item all the curves $C$ in fibres of $f$ are numerically propositional and $K_X\cdot C<0.$
\end{enumerate}
\end{defn}

Roughly speaking, an MFS $f\colon X\to Y$ has Fano fibres and relative Picard number one, i.e., $\rho(X)-\rho(Y)=1$ (see e.g., \cite[Theorem 3.2]{Mor82}. We call the morphism $f$ an \textit{extremal contraction} of $X$. For 3-dimensional MFSs, we have the following celebrated results due to Mori \cite[Theorem 3.5]{Mor82}.

	\begin{thm}\label{thm:mori}
		Let $ X $ be a $3$-dimensional smooth MFS, $f \colon X \to Y$ be an extremal contraction. Then $\dim Y \le 2$, and 
		\begin{enumerate}
			\item if $\dim Y=0$, then $X$ is a Fano manifold with Picard number $\rho(X)=1$;
			\item if $\dim Y=1$, then $Y$ is a smooth curve. Every fibre is irreducible and reduced, every smooth fibre $F$ is a del Pezzo surface and is diffeomorphic to one of the following
			$$\p^2, \ \p^1 \times \p^1, \ \p^2\#k\overline{\p^2} \ (3 \le k \le 8).$$
			Furthermore, if $F$ is diffeomorphic to $\p^2$, then $X$ is a $\p^2$-bundle over $Y$; if $F$ is diffeomorphic to $\p^1 \times \p^1$, then $X$ can be realized as a hypersurface in a $\p^3$-bundle over $Y$.
			\item if $\dim Y=2$, then $Y$ is a smooth surface and $f \colon X \to Y$ is a conic bundle. The smooth fibre $F$ is a conic curve, the singular fibre consists of two lines.
		\end{enumerate}
	\end{thm}

%\begin{lem} \label{lem:mori}  (Ref?)
%Let $f \colon X \to Y$ be a $3$-dimensional Mori fiber space. Then 
%\begin{enumerate}
%\item $f_* \colon \pi_1(X) \to \pi_1(Y)$ is surjective;
%\item $f_*\oo_X = \oo_Y$, $R^jf_* \oo_X=0$ for $j >0$;
%\item  if $\dim Y=2$, then the Picard numbers of $X$ and $Y$ satisfy $\rho(X)=\rho(Y)+1$;

%\end{enumerate}

%\end{lem}

\begin{lem}\label{lem:b2}
Let $f \colon X \to Y$ be a simply connected $3$-dimensional MFS with $\dim Y \ge 1$. Then $Y$ is simply connected and $b_2(X) = b_2(Y) +1$.
\end{lem}
\begin{proof}
The extremal contraction $f \colon X \to Y$ induces a surjective homomorphism $f_*\colon  \pi_1(X) \to \pi_1(Y)$ by \cite[Proposition 2.10.2]{Koll95}, therefore if $X$ is simply connected, so is $Y$.  

To prove the equality $b_2(X)=b_2(Y)+1$, consider the Leray-Serre spectral sequence of $f \colon X \to Y$. The $E_2$-terms are $E_2^{i,j}=H^i(Y, R^jf_* \oo_X)$ and the spectral sequence converges to $H^*(\oo_X)$. Since $f\colon X\to Y$ is a MFS, $-K_X$ is relatively ample. Hence by the relative Kodaira vanishing theorem (see e.g., \cite[Proposition 6.2]{MiPe97}), we have $R^jf_*\mathcal{O}_X=R^jf_*(-K_X+K_X)=0$ for $j>0$. Also, since $f$ has connected fibres, $f_*\mathcal{O}_X=\mathcal{O}_Y$. Therefore, the spectral sequence collapses at the $E_2$-page and one has an identity of Hodge numbers $h^{i,0}(X)=h^{i,0}(Y)$ ($i \ge 0$), which implies that $X$ and $Y$ have equal holomorphic Euler characteristics 
\begin{equation}\label{eq:holo}
\chi(\oo_X) =\chi(\oo_Y).
\end{equation}

If $h^{2,0}(X) =0$, then $h^{2,0}(Y)=0$. Also we have $H^1(\oo_X)=H^1(\oo_Y)=0$ since both $X$ and $Y$ are simply connected. From the exact sequence 
$$ \cdots \to H^1( \oo_X) \stackrel{\mathrm{exp}}{\longrightarrow} H^1( \oo_X^{*}) \to H^2(X,  \mathbb Z) \to H^2( \oo_X) \to \cdots$$
one deduces that the Picard number $\rho(X)$ equals to $b_2(X)$. Similarly $\rho(Y)=b_2(Y)$. It's known that the Picard numbers of $X$ and $Y$ satisfy $\rho(X)=\rho(Y)+1$ (see e.g., \cite[Theorem 3.2]{Mor82}). 

If $h^{2,0}(X) \ne 0$, then $h^{2,0}(Y) \ne 0$ and thus $\dim Y =2$, i.e., $f$ is of relative dimension 1. Hence $R^2f_*\mathcal{O}_X=0$.  It follows from \cite[12.1.3 Theorem)]{KoMo92} that $$b_2(X)-b_2(Y)=\dim_{\C}H_2(X/Y,\C)=1.$$ Here $H_2(X/Y,\C)\subset H_2(X,\C)$ is generated by the images of $H_2(X_y, \C)\to H_2(X,\C)$ where $X_y$ runs through all the fibres of $f$.
\end{proof}

\begin{lem}\label{lem:sign}
Let $f \colon X \to Y$ be a simply connected $3$-dimensional MFS, $\dim Y=2$. Then 
$$\mathrm{sign}(Y)=4 \chi(\oo_X)-e(Y).$$
\end{lem}
\begin{proof}
By the Hirzebruch Signature Theorem, the signature of $Y$ satisfies
$$\mathrm{sign}(Y)=\frac{1}{3}\langle p_1(Y), [Y] \rangle = \frac{1}{3}(\langle c_1(Y)^2, [Y] \rangle -2e(Y)).$$
By the Noether's formula (see e.g. \cite[I.14]{Beau96}), the holomorphic Euler characteristic of $Y$ satisfies 
$$\chi( \oo_Y)=\frac{1}{12}\langle c_1(Y)^2+c_2(Y), [Y] \rangle = \frac{1}{12}(\langle c_1(Y)^2,[Y]\rangle + e(Y)).$$
Together with the fact $\chi( \oo_X) = \chi( \oo_Y)$ (equation (\ref{eq:holo})), we have 
$$
\mathrm{sign}(Y)=4 \chi( \oo_X)-e(Y).
$$
\end{proof}

Simply connected $6$-manifolds with torsion free homology groups are classified by Wall (\cite{Wall66}) and Jupp (\cite{Jupp73}) in terms of a system of algebraic topological invariants. We rephrase the results in the following theorem which will be used in this paper. 

\begin{thm}\label{thm:6mfds}
Let $X$ and $X'$ be simply connected oriented smooth $6$-dimensional manifolds with torsion free homology. Then $X$ is oriented diffeomorphic to $X'$ if and only if $b_3(X) = b_3(X')$, and there is an isomorphism $\varphi \colon H^2(X, \Z) \to H^2(X', \Z)$, which preserves the cubic forms, i.e., for any $x$, $y$, $z \in H^2(X)$,
$$\langle x \cup y \cup z, [X] \rangle = \langle \varphi(x) \cup \varphi (y) \cup \varphi(z), [X']\rangle,$$ 
and preserves the characteristic classes, i.e., 
$$\varphi(w_2(X))=\varphi(w_2(X')), \ \ \varphi^*(Dp_1(X'))=Dp_1(X),$$ 
where $w_2(X)$ and $p_1(X)$ are the second Stiefel-Whitney class and the first Pontrjagin class of $X$, respectively, and $D$ stands for the Poincar\' e duality isomorphism.
\end{thm}

\section{The case $\dim Y=2$}\label{sec:2}
\subsection{The topology of projective bundles}\label{sec:projbdl}
Let $f \colon X \to Y$ be a simply connected $3$-dimensional MFS with torsion free homology and $\dim Y=2$.  If the discriminant curve of $f$ is empty, then $f$ is a holomorphic submersion. Thus $f \colon X \to Y$ is a fiber bundle (\cite{FiGr65}) with fiber $\p^1$ (by Theorem \ref{thm:mori}) and structure group the projective linear group $PGL_2(\mathbb C)$. The homotopy fibre of the natural map between classifying spaces $BGL_2(\mathbb C) \to BPGL_2(\mathbb C)$ is $B\mathbb C^* \simeq \mathbb C \mathrm P^{\infty}$ (since there is an exact sequence of Lie groups $0 \to \mathbb C^* \to GL_2(\mathbb C) \to PGL_2(\mathbb C) \to 0$). The obstructions to lifting the structure group of $f \colon X \to Y$ to $BGL_2(\mathbb C)$ are in the cohomology groups $H^i(Y; \pi_{i-1}(B\mathbb C^*))$, which all vanish.  Therefore, as a smooth fibre bundle,  $X$ is the projective bundle $\p(E)$ of a rank $2$ complex vector bundle $E$ over $Y$.  By Leray-Hirsch theorem, $H^*(X, \Z)$ is a free $H^*(Y, \Z)$-module with generators $1$ and $u$,  
\begin{equation}\label{eq:leray1}
H^*(X, \Z) = H^*(Y,\Z)\{1, u\}, 
\end{equation}
where $u=c_1(L)$ is the first Chern class of the tautological line bundle of $\p(E)$, and there is a relation 
\begin{equation}\label{eq:leray2}
u^2 - c_1u + c_2=0, 
\end{equation}
where $c_i=c_i(E)$ ($i=1,2$) is the $i$-th Chern class of $E$.

The tangent bundle $TX$ is isomorphic to $V \oplus f^*TY$, where $V = \ker (df \colon TX \to TY)$ is the vertical bundle, which under the identification of $X$ with $\p(E)$,  is isomorphic to the hom-bundle $Hom(L, L^{\perp})$. Therefore there is an isomorphism of vector bundles
$$TX \cong Hom(L, L^{\perp}) \oplus f^*TY.$$
There are isomorphisms of ${C}^{\infty}$ complex vector bundles
$$Hom(L, L^{\perp}) \oplus \underline{\mathbb C} \cong Hom(L, L^{\perp} \oplus L) \cong Hom(L, f^*E) \cong L^* \otimes f^*E,$$
where $\underline{\mathbb C}$ denotes the trivial ${C}^{\infty}$ line bundle, and $L^*$ is the dual bundle of $L$.
From this description the characteristic classes of $X$ are expressed in terms of that of $Y$ and $E$ as follows (cf.~\cite{BH60})
\begin{eqnarray}\label{eq:class}
c_1(X) & = & c_1(Y)+c_1-2u, \\
c_2(X) & = & c_2(Y)+(c_1-2u)c_1(Y), \\
c_3(X)& = & (c_1-2u)c_2(Y),\\
p_1(X) & = & \langle (c_1^2-4c_2)+p_1(Y), [Y] \rangle \sigma_Y
\end{eqnarray}
where $\sigma_Y$ is the cohomology fundamental class. From these 
we have 
\begin{equation}\label{eq:cp}
\langle c_1(X)p_1(X),[X] \rangle = -2\langle (c_1^2-4c_2)+p_1(Y), [Y] \rangle.
\end{equation}
On the other hand, by the Hirzebruch-Riemann-Roch formula, we have 
\begin{equation}\label{eq:cp1}
\langle c_1(X)p_1(X),[X] \rangle = \langle c_1(X)^3-2c_1(X)c_2(X), [X] \rangle = \langle c_1(X)^3, [X] \rangle -48 \chi( \oo_X).
\end{equation}
Therefore the characteristic numbers are subjected to the relation
\begin{equation}\label{eq:cchi}
\langle c_1(X)^3, [X] \rangle -48 \chi( \oo_X)=-2\langle (c_1^2-4c_2)+p_1(Y), [Y] \rangle.
\end{equation}
Moreover, from equation (\ref{eq:class}) we have 
\begin{equation}\label{eq:w2}
w_2(X)=w_2(Y)+w_2(E).
\end{equation} 

\subsection{The topology of singular Mori fiber spaces}
\begin{proposition}\label{prop:h2}
Let $f \colon X \to Y$ be a simply connected $3$-dimensional MFS with torsion free homology. Assume that $\dim Y=2$ and the discriminant curve of $f$ is nonempty. Then $c_1(X)$ is a primitive class. Let $u=c_1(X)-c_1(Y)$, then
$$H^2(X, \Z)= H^2(Y,\Z) \oplus \mathbb Z\langle u \rangle$$
$$2 \cdot H^4(X,\Z) =  H^4(Y,\Z) \oplus (\mathbb Z\langle u \rangle \cup H^2(Y,\Z)).$$
In particular, there exist $c_1 \in H^2(Y,\Z)$, $c_2 \in H^4(Y,\Z)$ such that $2u^2=c_1 u + c_2$.
\end{proposition}

\begin{proof}
Consider the Leray-Serre spectral sequence for the morphism $f \colon X \to Y$. The $E_2$-terms are $E_2^{p,q}=H^p(Y,R^qf_*(\mathbb Z))$ and the spectral sequence converges to $H^*(X,\mathbb Z)$ (see e.g. \cite[Corollary 2.3.4]{Dim04}). The stalk $(R^qf_*(\z))_y=H^q(f^{-1}(y), \Z)$ for any $y \in Y$, by the proper base change theorem (see e.g. \cite[Theorem 2.3.26]{Dim04}). In our situation $f^{-1}(y)$ is homeomorphic to $\p^1$ or $\p^1 \vee \p^1$  (Theorem \ref{thm:mori} (3)). Therefore $f_*(\z)=\z$, $R^1f_*(\z)=0$ are constant sheaves, and $R^2f_*(\z)$ is a sheaf of free abelian groups. The differential $d_3 \colon E_2^{0,2}=H^2(Y,R^2f_*(\z)) \to E_2^{3,0}=H^3(Y,f_*(\z))$ is trivial, since the target is the trivial group. We get a short exact sequence
$$0 \to H^2(Y,\Z) \stackrel{f^*}{\longrightarrow} H^2(X,\Z) \to H^0(Y, R^2f_*(\z)) \to 0.$$
In the sequel we will not distinguish a cohomology class in $H^*(X,\Z)$ and $H^*(Y,\Z)$.

The group $H^0(Y, R^2f_*(\z))$ is the group of global  sections of the sheaf, hence a free abelian group. Since $b_2(X) = b_2(Y) +1$, we know that $H^0(Y, R^2f_*(\z))$ is isomorphic to $\z$ and we have a short exact sequence 
\begin{equation}\label{eq:h2}
0 \to H^2(Y,\Z) \to H^2(X,\Z) \to \z \to 0.
\end{equation}
Let $F$ be a smooth fiber, $i \colon F \to X$ be the inclusion map. Then 
$$\langle i^*c_1(X), [F] \rangle = \langle c_1(F), [F] \rangle = \langle c_1(\p^1), [\p^1] \rangle=2.$$
Let $F'=f^{-1}(y)$ be a singular fiber over a smooth point $y$ in the discriminant curve, then $F'$ is a union of two lines $L_1$ and $L_2$.    Note that $F$ and $F'$ represent the same homology class in $H_2(X,\Z)$, since they are algebraically equivalent (see e.g. \cite[Definition 3.2.3, Section 3.3.4]{And04}). Let $c_1(X) = k \alpha$, where $\alpha \in H^2(X,\Z)$ is a primitive element and $k > 0$.  Then 
\begin{equation}\label{eq:c1}
\langle c_1(X), [F'] \rangle = k(\langle \alpha, [L_1] \rangle + \langle \alpha , [L_2] \rangle) =2.
\end{equation}
Note that both $L_1$ and $L_2$ represent non-trivial homology classes in $H_2(X,\Z)$ (since they have nonempty intersections with some hyperplanes, see e.g.  \cite[Theorem 1.2.19]{Laz04}), and the evaluation of $[L_i]$ with any classes in $\mathrm{Im}(f^* \colon H^2(Y,\Z) \to H^2(X,\Z))$ vanishes. From the exact sequence (\ref{eq:h2}) one sees that  $\langle \alpha, [L_i] \rangle$ must be positive integers. Then equation (\ref{eq:c1}) implies $k=1$,  $\langle \alpha , [L_i]\rangle =1$ and $c_1(X)$ is a primitive class. Furthermore, $[L_1] = [L_2]$. Let $u=c_1(X)-c_1(Y)$,  then $H^2(X,\Z) = H^2(Y,\Z) \oplus \z\langle u \rangle$.

Let $\{ y_1, \cdots, y_r\}$ be a basis of $H^2(Y,\Z)$, then $\{u, y_1, \cdots, y_r\}$ is a basis of $H^2(X,\Z)$. Let $\sigma_Y \in H^4(Y,\Z)$ be the fundamental class of $Y$.  By standard argument in algebraic topology (see Lemma \ref{lem:poincare}), $f^*(\sigma_Y)$ equals to the Poincar\' e dual of $[F]$. We identify $H^6(X,\Z)$ with $\z$ by the orientation. Then  
\begin{equation}\label{eq:u2}
u \cdot \sigma_Y=2
\end{equation}
There is a basis $\{z_1, \cdots, z_r\}$ of  $H^2(Y,\Z)$ such that $z_i \cdot y_j =\delta_{ij} \sigma_Y$.  Then $z_i \cdot u \cdot y_j=2\delta_{ij}$. The pairing 
$$-\cup - \colon H^2(X,\Z) \times H^4(X,\Z) \to \z$$
is unimodular, the equations $u \cdot \sigma_Y=2$ and $z_i \cdot u \cdot y_j=2\delta_{ij}$ imply that $\sigma_Y, uy_1, \cdots, uy_r $ generate the subgroup $2 \cdot H^4(X,\Z)$.
\end{proof}

From this proposition we see that the cup product structure on $H^2(X,\Z)$ is determined by the classes $c_1$ and $c_2$. We regard the singular curve $C_f \subset Y$, and the numerical invariants of $X$, such as $K_X^3$, and $\chi( \oo_X)$ as algebro-geometric data of the Mori fiber space $f \colon X \to Y$, and express the diffeomorphism invariants of $X$, such as the cubic form on $H^2(X,\Z)$, and the Pontrjagin class $p_1(X)$,  in terms of this data. This is done in the following proposition.

\begin{proposition}\label{prop:top2}
Let $f \colon X \to Y$ be a simply connected $3$-dimensional MFS with torsion free homology, $\dim Y=2$ and non-empty discriminant curve $C_f$. Let $u=c_1(X)-c_1(Y)$, $c_1 \in H^2(Y,\Z)$, $c_2 \in H^4(Y,\Z)$ such that $2u^2=c_1 u + c_2$ as in Proposition \ref{prop:h2}. Let $D \colon H^4(X,\Z) \to H_2(X,\Z)$ and $D_Y \colon H^2(Y,\Z) \to H_2(Y,\Z)$ be the Poincar\' e duality maps of $X$ and $Y$, respectively. Then
\begin{enumerate}
\item $f_*D(u^2)=-[C_f]$, $f_*D(p_1(X))=-3[C_f]$;
\item $b_3(X)=2(p_a(C_f)-1)= [C_f] \cdot [C_f] - \langle c_1(Y), [C_f]\rangle$, 
where $p_a(C_f)$ is the arithmetic genus of $C_f$;
\item $D_Y(c_1)=-[C_f]$;
\item $\langle c_2, [Y] \rangle = \langle c_1(X)^3, [X] \rangle + 3 e(X) -72 \chi(\oo_X) + \frac{5}{2} [C_f] \cdot [C_f]$;
\item From (1) we may write $p_1(X) = p + 3u^2$, where 
$p \in H^4(Y,\Z)$. Then 
$$\langle p, [Y] \rangle = 84\chi(\oo_X) - \langle c_1(X)^3, [X] \rangle -\frac{9}{2}e(X)-\frac{3}{2}b_3(X) -3[C_f] \cdot [C_f].$$
\end{enumerate}
\end{proposition}
\begin{proof}
(1) This follows from  \cite[Proposition 7.1.8 (iii)]{PS99} and \cite[Lemma 6]{ScTa20}.

(2) Since $f \colon X \to Y$ is a conic bundle (Theorem \ref{thm:mori} (3)),  by \cite[Lemma 7.1.10 (i)]{PS99}
$$e(X)=2(e(Y)-p_a(C_f)+1).$$
By  Lemma \ref{lem:b2}, $b_2(X)=b_2(Y)+1$. This shows $b_3(X)=2(p_a(C_f)-1)$. The arithmetic genus of $C_f$ equals $1+\frac{1}{2}([C_f] \cdot [C_f] - \langle c_1(Y), [C_f] \rangle )$ (see e.g. \cite[I.15]{Beau96}).

(3) For any $y \in H^2(Y,\Z)$, 
$$\langle f_*D(c_1u), y \rangle = \langle D(c_1u), y \rangle = \langle c_1 u y ,[X] \rangle = 2 \langle c_1y, [Y] \rangle$$
where the second equality follows from equation (\ref{eq:u2}). This shows that $f_*D(c_1u)=D_Y(2c_1)$. On the other hand, 
by the equations in (1),  and the fact $f_*D(c_2)=0$ (Lemma \ref{lem:poincare}) one has
$$-2[C_f] = f_*(D(2u^2))= f_*D(c_1u + c_2)= f_*D(c_1 u),$$
therefore  $D_Y(c_1)=-[C_f]$.

(4) Since $u=c_1(X)-c_1(Y)$, we have
\begin{equation}
\begin{aligned}
c_1(X)^3 & =  (u+c_1(Y))^3 \\
 & =  u^3 + 3u^2 c_1(Y) + 3uc_1(Y)^2 \\
 & =  u^3+3c_1c_1(Y)[Y] + 6c_1(Y)^2[Y]
 \end{aligned}
 \label{eq:c1X3}
\end{equation}
In which 
$$u^3=\frac{1}{2}(c_1u+c_2)u=\frac{1}{4}\langle (c_1^2+2c_2)u, [X] \rangle = \frac{1}{2}\langle c_1^2+2c_2,  [Y] \rangle = \langle c_2, [Y] \rangle + \frac{1}{2}[C_f] \cdot [C_f],$$
\begin{equation}\label{eq:c1c1Y}
\langle c_1c_1(Y), [Y]\rangle = \langle c_1(Y), D_Y(c_1) \rangle = \langle c_1(Y), -[C_f] \rangle = b_3(X)-[C_f] \cdot [C_f],
\end{equation}
\begin{equation}\label{eq:c1Y2}
\langle c_1(Y)^2,[Y]\rangle = 12\chi(Y, \oo_Y) -e(Y) = 12\chi(X, \oo_X)-b_2(X)-1.
\end{equation}
Therefore we get 
$$\langle c_1(X)^3, [X] \rangle = \langle c_2, [Y] \rangle -3e(X) + 72\chi(X, \oo_X) - \frac{5}{2} [C_f] \cdot [C_f].$$

(5) By Hirzebruch-Riemann-Roch formula 
$$\langle c_1(X) p_1(X), [X] \rangle =  \langle c_1(X)^3-2c_1(X)c_2(X), [X] \rangle =\langle c_1(X)^3, [X] \rangle -48 \chi( \oo_X).$$ 
On the other hand, 
\begin{eqnarray*}
\langle c_1(X) p_1(X),[X] \rangle & = & \langle (u+c_1(Y))(p+3u^2), [X]\rangle \\
& = & \langle u p +3u^3 + 3c_1(Y) u^2, [X] \rangle \\
& = & 2 \langle p , [Y] \rangle+ 3 \langle c_2, [Y] \rangle +  \frac{3}{2} [C_f] \cdot [C_f] + 3b_3(X)-3 [C_f] \cdot [C_f] 
\end{eqnarray*}
\end{proof}

\subsection{The classification}
For Mori fiber spaces $f \colon X \to Y$ and $f' \colon X' \to Y'$ with $\dim Y= \dim Y'=2$, if $f$ is a smooth fibration and $f'$ is singular, then by comparing Formula (\ref{eq:leray1}) and Proposition \ref{prop:h2} one see that their  cohomology rings are not isomorphic, hence $X$ and $X'$ are not diffeomorphic. This proves the first statement in Theorem \ref{thm:main} (4). To prove the main theorem for the case $\dim Y=2$, it suffices to consider smooth fibrations and singular fibrations separately. In the following we first study the classification of smooth fibrations and prove Theorem \ref{thm:main} (4a). Then we study the classifciation of singular fibrations and prove Theorem \ref{thm:main} (4b).
\begin{proposition}\label{lem:H2Y}
Let $f \colon X \to Y$, $f' \colon X' \to Y'$ be simply connected $3$-dimensional Mori fiber spaces with torsion free homology. Assume $b_2(Y)\ge 3$. Let $g \colon X' \to X$ be an orientation-preserving diffeomorphism, then $g^*|_{H^2(Y,\Z)} \colon H^2(Y,\Z) \to H^2(Y',\Z)$ is an isomorphism.
\end{proposition}
\begin{proof}
Consider the homomorphism $\Phi \colon H^2(X,\Z) \to \z$, $\Phi(\alpha)=\langle \alpha \cup \sigma_Y, [X] \rangle$, where $\sigma_Y$ is the fundamental class of $Y$. By Poincar\' e duality, $\Phi$ is non-trivial, and $H^2(Y,\Z)$ is a direct summand of $H^2(X,\Z)$ (formula (\ref{eq:leray1}) or Proposition \ref{prop:h2}) on which $\Phi$ vanishes. Therefore $H^2(Y,\Z) = \ker \Phi$. Similarly, we define $\Phi' \colon H^2(X',\Z) \to \z$ and have $H^2(Y',\Z) = \ker \Phi'$. The lemma claims that $g^*$ maps $\ker \Phi$ isomorphically to $\ker \Phi'$. To prove this statement, it suffices to show that $g^*(\sigma_Y)$ is a non-zero multiple of $\sigma_{Y'}$.

The intersection pairing $\cup \colon H^2(Y,\Z) \times H^2(Y,\Z) \to H^4(Y,\Z)$ is non-degenerate, the rank of a subgroup of $H^2(Y,\Z)$ on which the pairing vanishes is at most half of the rank of $H^2(Y,\Z)$. Now $g^*(H^2(Y,\Z)) \cap H^2(Y',\Z)$ is a subgroup of $H^2(Y',\Z)$ of rank at least $b_2(Y)-1 $. When $b_2(Y) \ge 3$, $b_2(Y)-1 > b_2(Y)/2$,  therefore there exist elements $\omega'_1$, $\omega'_2 \in g^*(H^2(Y,\Z)) \cap H^2(Y',\Z)$   such that $\omega'_1 \cup \omega'_2 \ne 0$. Let $\omega'_i = g^* \omega_i$ for $i=1,2$. Then $g^*(\omega_1 \cup \omega_2)=\omega'_1 \cup \omega'_2$, where $\omega_1 \cup \omega_2$ is a non-zero multiple of $\sigma_Y$, and $\omega'_1 \cup \omega'_2$ is a non-zero multiple of $\sigma_Y'$. 
\end{proof}

\begin{thm}\label{thm:dim2smooth}
Let $f \colon X \to Y$, $f' \colon X' \to Y'$ be a simply connected $3$-dimensional MFS with torsion free homology and $\dim Y =\dim Y'=2$, and empty discriminant curves. Then $X$ is oriented diffeomorphic to $X'$ if and only if they have the same $w_2$-type, $e(X)=e(X')$, and  one the following conditions are fulfilled
\begin{equation}\label{eq:chiK1}
 \chi(\oo_X) = \chi(\oo_{X'}), \ \ K_X^3=K_{X'}^3;
\end{equation}
or
\begin{equation}\label{eq:chiK2}
\chi(\oo_X) + \chi(\oo_{X'}) = \frac{1}{4} e(X), \ \ K_X^3 + K_{X'}^3=-12e(X). 
\end{equation}

These invariants satisfy the inequality $e(X) \ge 6\chi(\oo_X)$, with equality if and only if $b_2(X)=2$ (i.e., $Y=\p^2$). Moreover, if $b_2(X)$ is even or $e(X) \ge 12\chi(\oo_X)$, the case in equation (\ref{eq:chiK2}) cannot occur.
\end{thm}

\begin{proof}
We first show that these conditions are necessary. We begin with the case $b_2(Y) \ge 3$, so that by Proposition \ref{lem:H2Y} an orientation preserving diffeomorphism $X' \to X$ induces an isomorphism $H^2(Y,\z) \to H^2(Y',\z)$.

Let $x=2u-c_1$, where $u$ and $c_1$ are given in formulas (\ref{eq:leray1}) and (\ref{eq:leray2}), respectively. Let $\{ y_1, \cdots, y_r \}$ be a basis of $H^2(Y,\Z)$, then by formula (\ref{eq:leray1}), $x, y_1, \cdots, y_r$ generate an index $2$ subgroup $V$ of $H^2(X,\z)$, containing $H^2(Y,\z)$. The cubic form 
$$H^2(X,\z) \times H^2(X,\z) \times H^2(X,\z) \to \z$$
is determined by the following data 
\begin{equation}\label{eq:cubic}
x^3=-\langle 2c_1^2-8c_2, [Y] \rangle, \ \ x^2y_i=0, \ \ xy_iy_j = -2\langle y_iy_j, [Y] \rangle.
\end{equation}
The same applies to $Y'$.

Let $g \colon X' \to X$ be an orientation-preserving diffeomorphism, $g^* \colon H^2(X,\z) \to H^2(X',\z)$ be the induced isomorphism, then $g^*(V)$ is a subgroup of $H^2(X',\z)$ of index 2 containing $H^2(Y',\z)$, therefore $g^*(V)=V'$. Under the bases $\{x, y_1, \cdots, y_r\}$ and $\{x', y'_1, \cdots, y'_r\}$, the isomorphism $g^* \colon V \to V'$ is represented by the matrix 
$$\left ( \begin{array}{cc}
\pm 1 & b \\
0 & T 
\end{array} \right ),$$
where $b=(b_1, \cdots, b_r)$.
Then we have the following identity 
$$0= x^2y_j = (\pm x' + \sum b_iy_i)^2 g^*(y_j)=\pm 2x' \sum b_i y_i g^*(y_j).$$
This implies $\sum b_i y_i g^*(y_j)=0$ for all $1 \le j \le r$. Since $\{ g^*(y_i), \cdots, g^*(y_r)\}$  is a basis of $H^2(Y',\z)$, and the intersection form is non-degenerate, we have $\sum b_iy_i=0$. Therefore $g^*(x)=\pm x'$, and $g^* \colon H^2(Y,\z) \to H^2(Y',\z)$ is an isometry or anti-isometry. 

If $g^*(x)=x'$, then $g^*(x^3)=x'^3$,  and $g^* \colon H^2(Y,\z) \to H^2(Y',\z)$ is an isometry. Therefore $\mathrm{sign}(Y)=\mathrm{sign}(Y')$, which by Lemma \ref{lem:sign} implies $\chi(\oo_X) =\chi(\oo_{X'})$. From the identity $g^*(x^3)=x'^3$ we get $\langle c_1^2-4c_2, [Y] \rangle = \langle c_1'^2-4c_2', [Y'] \rangle$ (see equation (\ref{eq:cubic})). Together with $\chi(\oo_X) =\chi(\oo_{X'})$, and $\mathrm{sign}(Y)=\mathrm{sign}(Y')$, equation (\ref{eq:cchi}) implies that $\langle c_1(X)^3,[X]\rangle = \langle c_1(X')^3,[X']\rangle $.

Similarly, if $g^*(x)=-x'$, then $g^* \colon H^2(Y,\z) \to H^2(Y',\z)$ is an anti-isometry. Therefore $\mathrm{sign}(Y)=-\mathrm{sign}(Y')$, which by Lemma \ref{lem:sign} implies 
$$4(\chi(\oo_X) +\chi(\oo_{X'}))=2e(Y)=e(X).$$
From the identity $g^*(x^3)=-x'^3$ we get $\langle c_1^2-4c_2, [Y] \rangle = -\langle c_1^2-4c_2, [Y'] \rangle$. Then from equation (\ref{eq:cchi}) we have
$$\langle c_1(X)^3, [X] \rangle + \langle c_1(X')^3, [X'] \rangle =48(\chi(\oo_X) + \chi( \oo_{X'})) =12e(X).$$ 

Finally, recall $c_1(X)=c_1(Y) -x$ (see (\ref{eq:class})), therefore $w_2(X)-w_2(Y) \equiv x \pmod 2$.  Since $g^*w_2(X)=w_2(X')$, $g^*w_2(Y)=w_2(Y')$, $g^*(x)=\pm x'$, we see that the two Mori fiber spaces have the same $w_2$-type. 

Next we prove the necessary part for the case $b_2(Y) \le 2$. It is known that the only simply-connected compact complex surface with definite intersection form is $\p^2$ (\cite[p.376]{BHPV}). Therefore, if $b_2(Y)=b_2(Y')=1
$, then $\mathrm{sign}(Y)=\mathrm{sign}(Y')=1$; if $b_2(Y)=b_2(Y')=2
$, then $\mathrm{sign}(Y)=\mathrm{sign}(Y')=0$. Then by Lemma \ref{lem:sign}, $\chi(\mathcal O_X)=\chi(\mathcal O_X')=1$.

 If $p_1(X)$ and $p_1(X')$ are non-zero, then we define non-trivial homomorphisms
$$\Psi \colon H^2(X,\z) \to \z,  \ \ \Psi(a)=\langle a \cup p_1(X) ,[X] \rangle$$
and 
$$\Psi' \colon H^2(X',\z) \to \z,  \ \ \Psi'(a)=\langle a \cup p_1(X') ,[X'] \rangle.$$
By equation (\ref{eq:class}), we have $\ker \Psi = H^2(Y,\z)$, $\ker \Psi'=H^2(Y',\z)$. For a diffeomorphism $g \colon X' \to X$, $g^*p_1(X) = p_1(X')$, therefore we have $g^*|_{H^2(Y,\z)} \colon H^2(Y,\z) \to H^2(Y',\z)$ is an isomorphism. The argument for the preceding case still works and shows that $\langle c_1(X)^3,[X]\rangle = \langle c_1(X')^3,[X']\rangle $, and the two Mori fiber spaces have the same $w_2$-type.

If $p_1(X)=p_1(X')=0$, then by equation (\ref{eq:cp1}) and the fact $\chi( \mathcal O_X)=\chi( \mathcal O_X')$ we have $\langle c_1(X)^3,[X]\rangle = \langle c_1(X')^3,[X']\rangle $. It remains to show that the Mori fiber spaces have the same $w_2$-type.  If $b_2(Y)=1$, then $Y$ and $Y'$ are non-spin.  The assumption $p_1(X)=0=p_1(X')$ implies (equation (\ref{eq:class}) and (\ref{eq:w2})) 
$$w_2(X)-w_2(Y) \equiv c_1 \not \equiv 0 \pmod 2, \ \ w_2(X')-w_2(Y') \equiv  c_1 ' \not \equiv 0 \pmod 2,$$
therefore the two Mori fiber spaces are both of $w_2$-type $0$ (when $X$ and $X'$ are spin) or   $III_1$ (when $X$ and $X'$ are non-spin). In the case $b_2(Y)=2$, if both $Y$ and $Y'$ are spin, then the two Mori fiber spaces are both of $w_2$-type $0$ or $II$; if both $Y$ and $Y'$ are non-spin, under the standard basis of the intersection form of $Y$, we write $c_1=(a,b)$, then the assumption $p_1(X)=0$ implies $a^2-b^2 \equiv 0 \pmod 4$, therefore $a$ and $b$ are both even or odd; the same applies to $c_1'=(a',b')$. This shows that the two Mori fiber spaces are both of $w_2$-type $I$ or $III_0$. We finish the proof of the necessary part by excluding the case $Y$ spin and $Y'$ non-spin.

Let $x=2u-c_1$, $x'=2u'-c_1'$, $\{ y_1, y_2\}$ be the standard basis of the intersection form on $H^2(Y,\z)$, $\{y_1', y_2'\}$ be the standard basis of the intersection form on $H^2(Y',\z)$. By equations (\ref{eq:cubic}) and Poincar\' e duality, $x^2=y_1^2=y_2^2=0$. Assume there is an orientation-preserving diffeomorphism $g \colon X' \to X$, then the identities $g^*(x)^2=g^*(y_1)^2=g^*(y_2)^2=0$ (by a direct calculation) imply that $g^*(x)$, $g^*(y_1)$ and $g^*(y_2)$ equal either $mx'$ or $n(y_1' \pm y_2')$, $m$, $n \in \mathbb Z$. But this leads to a contradiction since $x$, $y_1$ and $y_2$ generate a subgroup of index $2$ in $H^2(X,\z)$, while the subgroup generated by $x'$, $y_1' + y_2'$ and $y_1' - y_2'$ is of index $4$ in $H^2(X',\z)$.

Now we prove the sufficient part. Assume the invariants of $X$ and $X'$ satisfy the identities stated in the theorem, we are going to show that there is an isomorphism between their cubic forms preserving the characteristic classes.  Then by Theorem \ref{thm:6mfds} they are oriented diffeomorphic. The identity $e(X)=e(X')$ implies $e(Y)=e(Y')$. By lemma \ref{lem:sign}, the identity $\chi(\oo_X) = \chi(\oo_{X'})$ implies $\mathrm{sign}(Y) =  \mathrm{sign}(Y')$, and the identity $\chi(\oo_X) = e(X')/4-\chi(\oo_{X'})$ implies $\mathrm{sign}(Y) =  -\mathrm{sign}(Y')$.  The intersection forms of $Y$ and $Y'$ are both even or odd since they have the same $w_2$-type. Tthe intersection forms of $Y$ and $Y'$ are indefinite, except for the case $Y=\p^2$. By the classification of indefinite symmetric unimodular forms (\cite{MH73}) they are isometric (if $\mathrm{sign}(Y) = \mathrm{sign}(Y')$) or anti-isometric (if $\mathrm{sign}(Y) =- \mathrm{sign}(Y')$). 

If the $w_2$-types are $0$, $I$ or $III_0$, then by equation (\ref{eq:w2}) we have $w_2(E) = 0$ or $w_2(E)=w_2(Y)$. Let $\varphi \colon H^2(Y,\z) \to H^2(Y',\z)$ be an isometry (or anti-isometry), then the mod $2$ reduction of $\varphi$ maps $w_2(Y)$ to $w_2(Y')$, since they are the characteristic element of the intersection forms. In other words, $\varphi(c_1(E)) - c_1(E') \equiv 0 \pmod 2$. 
If the $w_2$-types are $II$ or $III_1$, the conditions (\ref{eq:chiK1}) or (\ref{eq:chiK2}) imply 
$$\langle c_1^2-4c_2, [Y] \rangle = \langle c_1'^2-4c_2', [Y'] \rangle, \ \ \mathrm{or} \ \ \langle c_1^2-4c_2, [Y] \rangle + \langle c_1'^2-4c_2', [Y'] \rangle=0.$$
Therefore $\langle c_1^2, [Y] \rangle \equiv \langle c_1'^2, [Y'] \rangle \pmod 4$ or $\langle c_1^2, [Y] \rangle \equiv -\langle c_1'^2, [Y'] \rangle \pmod 4$.
By  Lemma \ref{lem:quadratic}, there exists an isometry (or anti-isometry) $\varphi \colon H^2(Y,\z) \to H^2(Y',\z)$ such that $\varphi(c_1(E)) - c_1(E') \equiv 0 \pmod 2$. 

Therefore in all cases we have an isometry or anti-isometry $\varphi \colon H^2(Y,\z) \to H^2(Y',\z)$ with $\varphi(c_1(E)) - c_1(E') \equiv 0 \pmod 2$. Denote $v=  (\varphi(c_1(E)) - c_1(E')) /2 \in H^2(Y')$, we extend $\varphi$ to an isomorphism $\varphi \colon H^2(X,\z) \to H^2(X',\z)$ by defining
$$\varphi(u)= \left \{ \begin{array}{ll}
u'+v & \textrm{if $\varphi$ is an isometry} \\
-u'+v &  \textrm{if $\varphi$ is an anti-isometry}
\end{array} \right. 
$$
Then $\varphi(x)= x'$ or $-x'$. If the condition (\ref{eq:chiK1}) holds, then by equation (\ref{eq:cp1}) one has
\begin{equation}\label{eq:cp2}
\langle c_1(X)p_1(X), [X] \rangle = \langle c_1(X')p_1(X'), [X'] \rangle.
\end{equation} 
Then equations (\ref{eq:cp}) and (\ref{eq:cubic}) imply that $\varphi$ is an isomorphism between the cubic forms. If the condition (\ref{eq:chiK2}) holds, then by equations (\ref{eq:cp1}) and (\ref{eq:cp}) one has $\langle c_1^2-4c_2, [Y] \rangle + \langle c_1'^2-4c_2', [Y'] \rangle=0$, which guarantees that $\varphi$ is an isomorphism between the cubic forms (note that in this case $\varphi(x)=-x'$).

 By equations (\ref{eq:class}), $\langle c_1(X)p_1(X), [X] \rangle = -\langle xp_1(X), [X] \rangle$. Therefore if the condition (\ref{eq:chiK1}) holds, from equation (\ref{eq:cp2})  one has $\langle xp_1(X), [X] \rangle = \langle x'p_1(X'), [X'] \rangle$. Let $\varphi_* \colon H_2(X',\z) \to H_2(X,\z)$ be the dual of $\varphi$, then 
$$\langle \varphi_*Dp_1(X'), x \rangle  =  \langle x' p_1(X'), [X'] \rangle = \langle xp_1(X), [X] \rangle = \langle Dp_1(X), x \rangle. $$
Note that for any $y \in H^2(Y,\z)$, $y' \in H^2(Y',\z)$, $\langle  y p_1(X), [X] \rangle = 0=\langle  y' p_1(X'), [X'] \rangle$, therefore  
$$\langle \varphi_*Dp_1(X'), y \rangle  =  \langle \varphi(y) p_1(X'), [X'] \rangle = 0 =  \langle Dp_1(X), y \rangle. $$
This shows that $\varphi_*Dp_1(X')=p_1(X)$. A similar calculation shows that under the condition (\ref{eq:chiK2}) the isomorphism $\varphi$ preserves the first Pontrjagin class as well. 

By construction, the mod $2$ reduction of $\varphi$ maps $w_2(Y)$ to $w_2(Y')$, $w_2(E)$ to $w_2(E')$, therefore  by equation (\ref{eq:w2}) it preserves the second Stiefel-Whitney class. Hence there is an orientation preserving diffeomorphism $g \colon X' \to X$ such that $g^*=\varphi$.

The restrictions of the invariants follow from the Miyaoka-Yau inequality and Noether's formula. Recall that we have $b_2(X)=b_2(Y)+1$ (Lemma \ref{lem:b2}) and $\chi(\oo_X)=\chi(\oo_Y)$ (equation (\ref{eq:holo})). If $b_2(X)>2,$ then the Miyaoka-Yaui inequality (Lemma \ref{lem:my}) implies $\langle 3c_2(Y), [Y] \rangle > \langle c_1^2(Y), [Y] \rangle$, which may be written as 
$$\langle 4c_2(Y), [Y] > \langle c_1^2(Y)+c_2(Y), [Y] \rangle.$$ 
By the Noether's formula, this inequality implies $e(Y)>3\chi(\oo_Y)$.  On the other hand, we have $e(X)=\sum_{\textnormal{even}}b_i(X)=2e(Y)$. This shows the inequality  $e(X)>6\chi(\oo_X)$.
If $b_2(X)=2$, then $Y=\p^2$, with $\chi(\oo_Y)=1$. Therefore $e(X)=6=6\chi(\oo_Y)$.

From the proof above one sees that there exist $X$ and $X'$ satisfy the condition (\ref{eq:chiK2}) only if there is a complex surface structure on $Y$ whose induced orientation is opposite to $Y$. This will not happen when $Y=\p^2$. Therefore in this case the Miyaoka-Yau inequality implies $e(Y) > 3 |\mathrm{sign}(Y)|$. Since $\mathrm{sign}(Y)=4\chi(\oo_X)-e(Y)$ (Lemma \ref{lem:sign}), a direct calculation shows $e(X) < 12 \chi(\oo_X)$. Furthermore,  the condition (\ref{eq:chiK2}) implies that $e(X)$ is divisible by $4$. Hence $b_2(X)$ must be an odd integer.

\end{proof}

\begin{thm}\label{thm:dim2singular}
Let $f \colon X \to Y$, $f' \colon X' \to Y'$ be simply connected $3$-dimensional MFSs with torsion free homology and $\dim Y =\dim Y'=2$, and non-empty discriminant curves $C_f$ and $C_{f'}$, respectively.  

Then if $X$ is oriented diffeomorphic to $X'$, then the two MFSs have the same $w_2$-type, $[C_f]$ and $[C_{f'}]$ have the same divisibility and type, $e(X)=e(X')$,  $b_3(X)=b_3(X')$ and one of the following conditions is fulfilled
\begin{equation}\label{eq:chiK31}
 \chi( \oo_X) = \chi( \oo_{X'}), \ \ K_X^3=K_{X'}^3, \ \textrm{$[C_f]$ and $[C_{f'}]$ have the same norm},
 \end{equation}
or
\begin{equation}\label{eq:chiK3}
\chi( \oo_X) + \chi( \oo_{X'}) = \frac{1}{4}(e(X)-b_3(X)), \ K_X^3+K_{X'}^3= -12e(X)+18b_3(X),
\end{equation}
$$\textrm{$[C_f]$ and $[C_{f'}]$ have opposite norm}.$$

Conversely, except in the case $\chi(\oo_X)=1$ and $b_2(X) \ge 10$, these conditions are sufficient for $X$ being oriented diffeomorphic to $X'$. In the exceptional range, these conditions determine the oriented diffeomorphism type only up to finitely many possibilities.

These invariants satisfy the inequality $e(X)+b_3(X) > 6\chi(\oo_X)$. Moreover, if $b_2(X)$ is even or  $e(X)+b_3(X) \ge 12\chi(\oo_X)$, the case of equation (\ref{eq:chiK3}) cannot occur.
\end{thm}

\begin{rem}
The conditions $\chi( \oo_X)=1$ and $b_2(X) \ge 10$ for the exceptional case come from the fact that, in this case, the intersection form of $Y$ is  isometric to $(1) \oplus q(-1)$, $q \ge 8$. For such quadratic forms, the action of the orthogonal group on primitive vectors with given norm and type is not transitive , but with finitely many orbit types (\cite[p.~337]{Wall62}).
\end{rem}

\begin{rem}
Notice that by Proposition \ref{prop:top2} (2), the condition $b_3(X)=b_3(X')$ can be replaced by $p_a(C_f)=p_a(C_{f'})$.
\end{rem}

\begin{proof}
We first show that these conditions are necessary.  First consider the case $b_2(Y) \ge 3$. By Lemma \ref{lem:H2Y}, an orientation preserving diffeomorphism $g \colon X' \to X$ induces an isomorphism between $H^2(Y,\z)$ and $H^2(Y',\z)$.

Let $x=4u-c_1$, let $\{ y_1, \cdots, y_r \}$ be a basis of $H^2(Y,\z)$, then by Proposition \ref{prop:h2}, $x, y_1, \cdots, y_r$ generate an index $4$ subgroup $V$ of $H^2(X,\z)$, containing $H^2(Y,\z)$. The cubic form 
$$H^2(X,\z) \times H^2(X,\z) \times H^2(X,\z) \to \z$$
is determined by the following data 
\begin{equation}\label{eq:cubic_sing}
x^3=8\langle 8c_2-c_1^2, [Y] \rangle, \ \ x^2y_i=0, \ \ xy_iy_j = -8\langle y_iy_j, [Y] \rangle.
\end{equation}
The same applies to $Y'$.

Let $g \colon X' \to X$ be an orientation-preserving diffeomorphism, $g^* \colon H^2(X,\z) \to H^2(X',\z)$ be the induced isomorphism, then $g^*(V)$ is a subgroup of $H^2(X',\z)$ of index $4$ containing $H^2(Y',\z)$, therefore $g^*(V)=V'$. Under the bases $\{x, y_1, \cdots, y_r\}$ and $\{x', y'_1, \cdots, y'_r\}$, the isomorphism $g^* \colon V \to V'$ is represented by the matrix 
$$\left ( \begin{array}{cc}
\pm 1 & b \\
0 & T 
\end{array} \right )$$
We have $g^*(x)=\pm x' + \sum b_iy_i$, then 
$$0= x^2y_i = (\pm x' + w)^2 g^*(y_i)=\pm 2x'w g^*(y_i)$$
implies $w g^*(t_i)=0$ for all $1 \le i \le r$. Since $\{ g^*(y_i), \cdots, g^*(y_r)\}$  is a basis of $H^2(Y',\z)$, and the intersection form is non-degenerate, we have $w=0$. Therefore $g^*(x)=\pm x'$, and $g^* \colon H^2(Y,\z) \to H^2(Y',\z)$ is an isometry or anti-isometry. The isomorphism $g_* \colon H_2(X',\z) \to H_2(X,\z)$ descends to an isomorphism $g_* \colon H_2(Y',\z) \to H_2(Y,\z)$. Since $g^*(p_1(X))=p_1(X')$, by Proposition \ref{prop:top2} (1), $g_*([C_f'])=[C_f]$. Hence $[C_f]$ and $[C_f']$ have the same type and divisibility.  

If $g^*(x)=x'$, then $g^* \colon H^2(Y,\z) \to H^2(Y',\z)$ is an isometry. Therefore $\mathrm{sign}(Y)=\mathrm{sign}(Y')$, hence $\chi(\oo_X) =\chi(\oo_{X'})$ (Lemma \ref{lem:sign}). The identity $g_*([C_f'])=[C_f]$ implies $[C_f]$ and $[C_f']$ have the same norm, and  $g^*(c_1)=c_1'$ (Proposition \ref{prop:top2} (3)). Therefore $g^*(u)=u'$, and hence $\langle u^3, [X] \rangle =  \langle u'^3, [X'] \rangle$. By equations (\ref{eq:c1c1Y}) and (\ref{eq:c1Y2}) we have  
$$\langle c_1c_1(Y), [Y]\rangle = \langle c_1'c_1(Y'), [Y']\rangle , \ \ \langle c_1(Y)^2, [Y]\rangle = \langle c_1(Y')^2, [Y'] \rangle .$$ 
Therefore by equation (\ref{eq:c1X3}) $\langle c_1(X)^3, [X] \rangle =\langle c_1(X')^3, [X'] \rangle $. Hence $K_X^3=K_{X'}^3$.

If $g^*(x)=-x'$, then $g^* \colon H^2(Y,\z) \to H^2(Y',\z)$ is an anti-isometry. Therefore $\mathrm{sign}(Y)=-\mathrm{sign}(Y')$. By Lemma \ref{lem:sign} we have $$4(\chi(\oo_X) + \chi(\oo_{X'}))=2e(Y)=e(X)-b_3(X).$$ 
The identity $g_*([C_f'])=[C_f]$ implies that $[C_f]$ and $[C_f']$ have opposite norm, and $g^*(c_1)=-c_1'$. Therefore $g^*(u)=-u'$, $\langle u^3, [X] \rangle = - \langle u'^3, [X'] \rangle$. Equations  (\ref{eq:c1c1Y}) and (\ref{eq:c1Y2}) imply
$$\langle c_1c_1(Y), [Y] \rangle + \langle c_1'c_1(Y'), [Y']\rangle = 2b_3(X),$$
$$\langle c_1(Y)^2, [Y] \rangle + \langle c_1(Y')^2, [Y'] \rangle = 3e(X)-3b_3(X)-2b_2(X)-2.$$
Therefore by equation (\ref{eq:c1X3}) we have
$$\langle c_1(X)^3, [X] \rangle + \langle c_1(X')^3, [X'] \rangle = 12e(X)-18b_3(X).$$

Finally, recall $u = c_1(X)-c_1(Y)$, therefore $w_2(X)-w_2(Y) \equiv u \pmod 2$. Since $g^*w_2(X)=w_2(X')$, $g^*w_2(Y)=w_2(Y')$, $g^*(u)=\pm u'$, we see that the two Mori fiber spaces have the same $w_2$-type. 

The proof for the case $b_2(Y) \le 2$ is by elementary but lengthy calculation. We give a sketch of the proof here. The reader may check the details in the Appendix.

\textbf{The case $b_2(X)=2$.} In this case we have $b_2(Y)=b_2(Y')=1$,  $\mathrm{sign}(Y)=\mathrm{sign}(Y')=1$ and $\chi(\oo_X)=\chi(\oo_X')=1$.  As in Proposition \ref{prop:h2}, let $u=c_1(X)-c_1(Y)$, $y \in H^2(Y,\z)$ be a generator, then $u$ and $y$ form a basis of $H^2(X,\z)$, with $2u^2=c_1uy + c_2y^2$, $c_1$, $c_2 \in \z$. Note that the Poincar\' e dual of  $[C_f]$ is $-c_1y$ (Proposition \ref{prop:top2}), where $C_f$ is an algebraic curve in $Y$. Therefore $c_1 \ne 0$. Since $Y$ is non-spin and $u$ is a primitive element, by definition, the Mori fiber space $f \colon X \to Y$ has $w_2$-type $III_1$. The same applies to $X'$. It remains to show $\langle c_1(X)^3, [X] \rangle =  \langle c_1(X')^3, [X'] \rangle$. By Proposition \ref{prop:top2} (4), it suffices to show that $c_1=c_1'$, $c_2=c_2'$.

By Proposition \ref{prop:top2} (2), and the fact $c_1(Y)=-3y$, we have $b_3(X)=c_1^2+3c_1$. Therefore if $X$ and $X'$ are diffeomorphic, then $c_1^2+3c_1=c_1'^2+3c_1'$, from which we get $c_1=c_1'$. 

By Proposition \ref{prop:top2} (5),  $p_1(X)=py^2 + 3u^2$, $p \in \z$. Let $\Psi \colon H^2(X,\mathbb Q) \to \mathbb Q$ be the homomorphism, $\Psi(a)= \langle a \cup p_1(X), [X] \rangle$. Since $\Psi(y) = \langle y \cup p_1(X), [X] \rangle = 3c_1 \ne 0$,  we may take $v = u + \beta y$, $\beta \in \mathbb Q$, as a basis of $\ker(\Psi)$. Then $v$ and $y$ form a basis of $H^2(X,\mathbb Q)$.
Let $g \colon X' \to X$ be an orientation preserving diffeomorphism. Since $\Psi(y)=3c_1=3c_1'=\Psi'(y')$, the induced isomorphism $g^* \colon H^2(X,\mathbb Q) \to H^2(X',\mathbb Q)$ satisfies 
\begin{equation}\label{eq:g}
g^*(v) = \lambda v', \ \ g^*(y)= y' + av', \ \lambda,  a \in \mathbb Q.
\end{equation}
Note that the transformation matrices between the bases $\{u,y\}$ and $\{v,y\}$, $\{u', y'\}$ and $\{v', y'\}$ are unimodular, and the matrix representing $g^* \colon H^2(X,\z) \to H^2(X',\z)$ under the bases $\{u,y\}$ and $\{u', y'\}$ is unimodular, therefore the matrix representing (\ref{eq:g}) is unimodular, which means $\lambda = \pm 1$. Moreover, we may assume $a \ne 0$, otherwise we are in the situation where $g^*(H^2(Y,\z) )= H^2(Y',\z)$, which has been handled. In the end, using the fact that $g^*$ is an isomorphism between the cubic forms one may deduce $c_2=c_2'$. 

\textbf{The case $b_2(X)=3$}. In this case $b_2(Y)=b_2(Y')=2$. The Hodge structure of $Y$ implies that $\chi(\oo_Y)=1$, and $\langle c_1(Y)^2, [Y] \rangle =8$. The intersection forms of $Y$ and $Y'$ are indefinite, hence are both hyperbolic forms. Let $\{ e_1, e_2\}$ and $\{e_1', e_2'\}$ be a symplectic basis of $H^2(Y, \mathbb Q)$ and $H^2(Y', \mathbb Q)$, respectively. Denote $x=4u-c_1$, $x'=4u'-c_1'$, then $\{ x, e_1, e_2 \}$ and $\{ x' , e_1', e_2' \}$ form a basis of $H^2(X, \mathbb Q)$ and $H^2(X', \mathbb Q)$, respectively. Let $g \colon X' \to X$ be an orientation-preserving diffeomorphism. If there exists some $c \in g^*(H^2(Y,\z)) \cap H^2(Y', \z)$ with $c^2 \ne 0$, then the statement of Proposition \ref{lem:H2Y} is still valid and the proof for the case $b_2(X) >3$ works. Therefore in the following we assume that $g^*$ maps the Lagrangian of $H^2(Y, \z)$ to the Lagrangian of $H^2(Y', \z)$. By choosing the symplectic bases appropriately, we may assume $g^*(e_1)=e_1'$. Using the fact that $g^*$ is a ring isomorphism, we get either $g^*(H^2(Y, \z))=H^2(Y',\z)$, which has been handled, or $g^*$ has the form  
$$g^*(e_1)=e_1', \ g^*(e_2) = \alpha x', \ g^*(x)=\beta e_2'$$
with $\alpha \beta =1$. Denote $c_1=a e_1 + b e_2$, $c_1'=a'e_1'+b'e_2'$. Then from the expression of the Pontrjagin class $p_1(X)=pe_1e_2+3u^2$ (Proposition \ref{prop:top2}) we have $b=b'=0$.  From the equation
$$1=\chi(\oo_Y)=\chi(\oo_X)=\frac{1}{48}(c_1(X)^3-c_1(X)p_1(X))=\frac{1}{48}(2p+48)$$
we get $p=0$. The same holds for $p_1(X')$.  This leads to a contradiction between the fact $g^*(p_1(X))=p_1(X')$ and the cohomology ring structure of $X$ and $X'$.

Now we turn to the proof of the sufficient part. Assume that $X$ and $X'$ have equal invariants stated in the theorem. Then $H^2(Y,\z)$ and $H^2(Y',\z)$ have the same rank. We first assume that the condition (\ref{eq:chiK31}) holds. Then $\mathrm{sign}(Y)=\mathrm{sign}(Y')$. Together with the assumption that they have the same $w_2$-type, one sees that  the intersection forms of $Y$ and $Y'$ are isometric. Since $[C_f]$ and $[C_{f'}]$ have the same divisibility, type and norm, so do $c_1$ and $c_1'$ (Proposition \ref{prop:top2} (3)). By a theorem of Wall (Theorem \ref{thm:quadr}), there is an isometry $\varphi \colon H^2(Y,\z) \to H^2(Y',\z)$ such that $\varphi(c_1)=c_1'$.
We extend $\varphi$ to an isomorphism $\varphi \colon H^2(X,\z) \to H^2(X',\z)$ by letting $\varphi(u)=u'$. Then $\varphi(x)=x'$. By Proposition \ref{prop:top2} (4), $\langle c_2, [Y] \rangle = \langle c_2', [Y'] \rangle$.
Therefore we see from equations (\ref{eq:cubic_sing}) that $\varphi$ preserves the cubic forms. Let $\varphi_* \colon H_2(X',\z) \to H_2(X,\z)$ be the dual of $\varphi'$. Then by Proposition \ref{prop:top2} (5), 
$$\langle \varphi_*Dp_1(X'), u \rangle = 3\langle u'^3,[X'] \rangle + \langle u'p', [X'] \rangle = 3\langle u'^3,[X'] \rangle + 2 \langle p', [Y'] \rangle.$$ 
Similarly, $\langle Dp_1(X), u \rangle =  3\langle u^3,[X] \rangle + 2 \langle p, [Y] \rangle$. Therefore $\langle \varphi_*Dp_1(X'), u \rangle  = \langle Dp_1(X), u \rangle$. 
For any $y \in H^2(Y,\z)$, 
$$\langle \varphi_*Dp_1(X'), y \rangle = \langle (p'+3u'^2) \varphi(y), [X'] \rangle= 3\langle \varphi(u)^2 \varphi(y) , [X'] \rangle , $$
whereas $$\langle Dp_1(X), y \rangle = 3\langle u^2y , [X] \rangle.$$
Therefore $\varphi_*Dp_1(X')=p_1(X)$. Since $w_2(X) \equiv u + w_2(Y) \pmod 2$, the mod $2$ reduction of $\varphi$ maps $w_2(X)$ to $w_2(X')$. By Theorem \ref{thm:6mfds} there is an orientation preserving diffeomorphism  $g \colon X' \to X$ such that $g^*=\varphi$.

The proof for the case where the condition (\ref{eq:chiK3}) holds is similar. 

The proof of the inequality $e(X)+b_3(X) > 6\chi(\oo_X)$ and the exclusion of the possibility described by condition (\ref{eq:chiK3}) is identical with the proof in Theorem \ref{thm:dim2smooth}, with the only modification that in this case $e(X)+b_3(X)=\sum_{\textnormal{even}}b_i(X)=2e(Y)$.
\end{proof}

\section{The case $\dim Y=1$}\label{sec:1}
In the first part of this section we prove the classification theorem for MFS $f \colon X \to Y$ with $\dim Y = 1$. In the second part we show in Theorem \ref{thm:mixed} that if a MFS $X$ with $\dim Y=2$ is diffeomorphic to another MFS $X'$ with $\dim Y'=1$ then $X$ is diffeomorphic to $\p^2 \times \p^1$ or the Fano $3$-fold $\mathbb X$ described in the introduction. This proves Theorem \ref{thm:main} (1).

\subsection{The topology of Mori fiber spaces}
Assume $f \colon X \to Y$ is a simply-connected $3$-dimensional MFS with $\dim Y=1$ and $\mathrm{tors}H^3(X,\z)=0$. By \cite[Proposition 2.10.2]{Koll95}, $Y$ is simply connected, hence $Y=\p^1$. Note that in this case  (by Lemma \ref{lem:b2})
$$b_2(X)=b_2(Y)+1=2, \ \ \chi(\mathcal O_X)=\chi( \mathcal O_{\p^1})=1.$$
We begin with the  analysis of the cohomology group of $X$.
\begin{proposition}\label{prop:y1h2}
 $f^* \colon H^2(Y,\z) \to H^2(X,\z)$ is injective.  Except for the case where the general fiber of $f$ is a del Pezzo surface diffeomorphic to $\p^2 \# 3\overline{\p^2}$,  $f^*H^2(Y,\z)$ is a direct summand of $H^2(X, \z)$.
\end{proposition}
\begin{proof}
Consider the Leray spectral sequence 
$$E^{p,q}_2=H^p(Y,R^qf_*(\z)) \Rightarrow H^{p+q}(X,\z).$$
By Theorem \ref{thm:mori} (2), the fiber of $f$ is irreducible and reduced, in particular it is connected. Also, by the proper base change theorem (see e.g. \cite[Theorem 2.3.26]{Dim04}), the stalk $f_*(\z)_p =H^0(f^{-1}(p),\z) = \z$. Hence $f_*(\z) = \z$, since $Y\cong\mathbb{P}^1$.  
It's known that $f^* \colon H^2(Y, \q) \to H^2(X,\q)$ is injective (\cite[Lemma 7.28]{Voi02}).  Also, since $H^2(Y,\mathbb{Z})$ and $H^2(X, \mathbb{Z})$ are torsion free, we know that $f^*\colon  H^2(Y, \z) \to H^2(X,\z)$ is injective.  This implies that the differential $d_2 \colon E^{0,1}_2 \to E^{2,0}_2$ is trivial. Therefore $E^{0,1}_2 =H^0(Y, R^1f_*(\z))= E^{0,1}_{\infty}=0$ (the last equality is due to the simply-connectedness of $X$). The stalk of the sheaf $(R^1f_*(\z))_p =H^1(f^{-1}(p),\z)$ is trivial when $p$ is a regular value of $f$. Therefore the sheaf $R^1f_*(\z)$ is supported on the set of discriminant points of $f$, which is a finite set. Then the vanishing of $H^0(Y,R^1f_*(\z))$ implies that $R^1f_*(\z)=0$.

Therefore the spectral sequence reduces to a short exact sequence
$$0 \to H^2(Y,\z) \stackrel{f^*}{\rightarrow} H^2(X,\z) \to H^0(Y, R^2f_*(\z)) \to 0.$$
The stalk of $R^2f_*(\z)$ at $p$ is isomorphic to $H^2(f^{-1}(p),\z)$. When $p$ is a regular point of $f$, this group is torsion free. When $p$ is a discriminant point of $f$,  the torsion subgroup of $H^2(f^{-1}(p),\z)$ is isomorphic to the torsion subgroup of $H_1(f^{-1}(p),\z)$. Therefore the order of the torsion subgroup of $H^0(Y,R^2f_*(\z))$ equals to the product of the order of torsion subgroups of $H_1(F_s,\z)$ of all singular fibers $F_s$. 

In the following we show that $H_1(F_s,\z)$ is torsion free, except for the case that $K_F^2=K_{F_s}^2=6$, where $F$ is the general fibre of $f$. In fact, by \cite[Section 4, (4.1)-(4.4))]{Fuj90}, we know that the singular fibres $F_s$
have isolated singularities when $1\leq K_F^2=K_{F_s}^2\leq 4$. Note that since $X, Y$ are smooth, $F_s$ is a complete intersection scheme. Since $F_s$ is reduced and irreducible by Theorem \ref{thm:mori} (2), $F_s$ is normal by \cite[Proposition 8.23]{Har77}. Also, by \cite[Theorem 3.1]{Fuj90}, $F_s$ is again normal if $K_F^2=K_{F_s}^2=5$ or $K_F^2=K_{F_s}^2>6$. Note that $F_s$ is  Gorenstein since $F_s$ is a complete intersection variety (see e.g. \cite[Corollary 21.19]{Eis95}). Now in the case that $K_F^2=K_{F_s}^2\neq6$, by \cite[Theorem 2.2]{HiWa81}, we know that $F_s$ is either rational or obtained by contracting the section of a $\mathbb{P}^1$-bundle over an elliptic curve. When $F_s$ is rational, so is its desingularization $\widetilde{F}_s$. Hence the fundamental group $\widetilde{F}_s$ is trivial.  Moreover, by \cite[Theorem 2.1]{ADH16}, $\pi_1(\widetilde{F}_s)\to \pi_1(F_s)$ is surjective. Hence $F_s$ is simply connected.
In the remaining case, $F_s$ is obtained by contracting the section of a $\mathbb{P}^1$-bundle over an elliptic curve. Denote the $\mathbb{P}^1$-bundle over the elliptic curve $C$ by $\widehat{F}_s$, we then have a long exact sequence of homotopy groups $$\{1\}=\pi_1(\mathbb{P}^1)\to \pi_1(\widehat{F}_s)\to\pi_1(C)\to \{1\}.$$ Hence  $\pi_1(\widehat{F}_s)=\pi_1(C)$.  After contracting the section of $\widehat{F}_s\to C$, which is connected, all elements of $\pi_1(\widehat{F}_s)=\pi_1(C)$ are killed. Thus $F_s$ is again simply connected. In summary, we get $F_s$ is simply connected when $K_F^2=K_{F_s}^2\neq6$.
\end{proof}

%{\color{red} In the case that $K_{F_s}^2=6$, it is likely that $H_1(F_s,\z)$ has torsion (see e.g., \cite[Section 1.3]{Reid94} for some examples).} 
Next we consider the characteristic classes. Let $\sigma_Y \in H^2(Y,\z)$ be the fundamental class, then $f^*(\sigma_Y)=d(X,Y) \cdot y$ with $y$ a primitive class and $d(X,Y) > 0$. By the above proposition, $d(X,Y)=1$ when the smooth fiber of $f$ is not diffeomorphic to $\p^2 \# 3 \overline{\p^2}$. Choose $x \in H^2(X,\z)$ such that $x$ and $y$ form a basis of $H^2(X,\z)$, and $c_1(X) = kx + ly$ with $k \ge 0$. Note that the integer $k$ is independent of the choice of $x$.

\begin{lem}\label{lem:fiberk} 
Let $F$ be a smooth fibre of $f$. Then
$$k = \left \{ \begin{array}{cl}
1, & \text{if $F$ is diffeomorphic to $\p^2 \# r \overline{\p^2}$} \\
2, & \text{if $F$ is difffeomorphic to $\p^1 \times \p^1$} \\
3, & \text{if $F$ is diffeomorphic to $\p^2$.}
\end{array} \right.$$
\end{lem}
\begin{proof}
Let $i \colon F \to X$ be the inclusion map, then $i^*y=0$. The normal bundle of $F$ in $X$ is trivial, therefore $c_1(F)=i^*c_1(X)=ki^*x$. If the smooth fiber of $f$ is $\p^2$, then $X=\p(E)$ is a projective bundle over $Y$ (Theorem \ref{thm:mori}). By Leray-Hirsch theorem, we may take $x=c_1(L)$, where $L$ is the tautological line bundle over $X$. Therefore $i^*(x)$ is the first Chern class of the tautological line bundle over $\p^2$. This shows $k=3$. 

If $F= \p^2 \# r \overline{\p^2}$, by Hirzebruch Signature Theorem, we have 
$$3(1-r)=3\mathrm{sign}(F)= \langle p_1(F), [F] \rangle =\langle c_1(F)^2 -2c_2(F), [F] \rangle.$$
Since $\langle c_2(F), [F] \rangle = \chi(F)=r+3$, we have $k^2 \langle i^*x^2, [F] \rangle = 9-r $, which is an integer between $1$ and $6$ for $3 \le r \le 8$.  This implies $k=1$ or $2$. The fact that $\p \# r\overline{\p^2}$ is non-spin rules out the possibility of $k=2$.

If $F=\p^1 \times \p^1$, then $X$ is a quadric hypersurface in a $\p^3$-bundle over $Y=\p^1$. Let $\xi$ be a rank $4$ holomorphic vector bundle over $Y$, $\pi \colon \p(\xi) \to Y$ be the projective bundle,  $j \colon X \to \p(\xi)$ be the inclusion map, $f = \pi \circ j \colon X \to Y$ be a quadric hypersurface bundle with fiber $F=f^{-1}(y)$ a quadric surface in $\pi^{-1}(y)=\p^3$. 

Let $L^*$ be the dual of the tautological line bundle over $\p(\xi)$, $u=c_1(L^*)$, $x=j^*u$, then the restriction of $x$ to a fiber $F$ equals to the restriction of the hyperplane class in $\p^3$ to the quadric surface $F$, which is a primitive class. Also note that the direct summand $H^2(Y,\z)$ equals to $\ker\{ i^* \colon H^2(X,\z) \to H^2(F,\z)\}$. Therefore $\{x , y\}$ is a basis of $H^2(X,\z)$. Now by Hirzebruch Signature Theorem, we have
$$0=3\mathrm{sign}(F)= \langle p_1(F), [F] \rangle =\langle c_1(F)^2 -2c_2(F), [F] \rangle.$$
On the other hand, $\langle c_2(F), [F] \rangle = \chi(F)=4$, therefore $k^2 \langle i^*x^2, [F] \rangle = 8$. This implies $k=1$ or $2$. But $i^*x$ is a primitive class, and $F$ is spin, therefore $k=2$.
\end{proof}

\begin{lem}\label{lem:k}
Let $F$ be a smooth fiber of $f$, $\{x, y\}$ be the basis of $H^2(X,\z)$ chosen as above. Then 
\begin{equation}\label{eq:x2y}
\langle x^2y, [X] \rangle = \frac{\langle c_1(F)^2, [F] \rangle }{k^2 d(X,Y)} = \left \{ \begin{array}{cl}
1, & k=3 \\
2, & k=2 \\
(9-r)/d(X,Y), & k=1
\end{array} \right.
\end{equation}
\end{lem}
\begin{proof}
Note that the Poincar\'e dual of $\sigma_Y$ is $[F]$ (Lemma \ref{lem:poincare}), therefore
$$\langle x^2y, [X]\rangle = \langle i^*x^2, [F] \rangle/d(X,Y) = \frac{\langle i^*c_1(X)^2, [F] \rangle}{k^2 d(X,Y)} = \frac{\langle c_1(F)^2, [F] \rangle}{k^2d(X,Y)}.$$ 
Then equation (\ref{eq:x2y}) follows from Lemma \ref{lem:fiberk}.
\end{proof}

\begin{lem}\label{lem:indep}
The cohomology classes $x^2$ and $xy$ are linearly independent in $H^4(X,\z)$.
\end{lem}
\begin{proof}
Assume $a x^2 + b xy =0$ for some integers $a$, $b$. Restricting to the fiber $F$ we have $ai^*x^2=0$. But $\langle i^*x^2, [F] \rangle = d(X,Y) \langle x^2y, [X] \rangle$ is non-zero by equation (\ref{eq:x2y}), thus $a=0$ and $bxy=0$. By equation (\ref{eq:x2y}) $xy \ne 0$, therefore $b=0$. 
\end{proof}

\begin{proposition}\label{lem:gh2}
Let $g \colon X' \to X$ be an orientation-preserving diffeomorphism. Let $\{x, y\}$ and $\{x',y'\}$ be the basis of $H^2(X,\z)$ and $H^2(X',\z)$ chosen as above, respectively. Then $g^*(x)=\pm x' + ny$, $g^*(y)=y'$ for some $n \in \mathbb Z$. 
\end{proposition}
\begin{proof}
Assume $g^*(x)=a_1x' + a_2y'$, $g^*(y)=b_1x' + b_2y'$. Since $y^2=0$, $y'^2=0$, we have 
$$0=g^*(y^2)=b_1^2x'^2+2b_1b_2x'y'.$$
By Lemma \ref{lem:indep} one gets $b_1=0$. Therefore $a_1=\pm 1$, $b_2=\pm 1$. If $b_2=-1$, then $\langle x^2y, [X]\rangle=-\langle x'^2y', [X']\rangle$, which leads to a contradiction to equation (\ref{eq:x2y}). This shows $b_2=1$.
\end{proof}

\begin{lem}\label{lem:s}
Let $\{x, y\}$ be a basis of $H^2(X,\z)$ as above. By Lemma \ref{lem:indep}, we may write the first Pontrjagin class of $X$ as $p_1(X)=sx^2 + txy$, with $s, t \in \mathbb Q$. Then 
$$s=\frac{3k^2\mathrm{sign}(F)}{\langle c_1(F)^2, [F] \rangle}.$$
\end{lem}
\begin{proof}
By Hirzebruch Signature Theorem we have 
\begin{equation}\label{eq:signF}
3 \mathrm{sign}(F)=\langle p_1(F), [F] \rangle = \langle i^*p_1(X), [F] \rangle = \langle s (i^*x)^2, [F] \rangle =\frac{s}{k^2} \langle c_1(F)^2  , [F] \rangle ,
\end{equation}
where the last identity follows from $c_1(X)=kx + ly$.
\end{proof}

We summarize the topological invariants of the smooth fiber $F$ in the following table. 
\begin{equation}\label{eq:table}
\begin{array}{c|ccc}
F & \p^2 & \p^1 \times \p^1 & \p^2\#r\overline{\p^2} \ (3 \le r \le 8 ) \\
\hline
e(F) & 3 & 4 & r+3\\
\mathrm{sign}(F) & 1 & 0 & 1-r \\
\langle c_1(F)^2, [F] \rangle & 9 & 8 & 9-r \\
k & 3 & 2 & 1 \\
s & 3 & 0 & 3(1-r)/(9-r)
\end{array}
\end{equation}

Define 
$$K(X,Y)=\frac{d(X,Y)}{6}(K_{X/Y}^3-K_X^3).$$
Then a direct calculation shows that (see Lemma \ref{lem:k})
\begin{equation}\label{eqn:k}
K(X,Y)= \langle c_1(F)^2, [F] \rangle = \left \{ \begin{array}{cl}
	9, & \text{if $F \cong \p^2$} \\
	8, & \text{if $F \cong \p^1 \times \p^1$} \\
	9-r, & \text{if $F \cong \p^2 \# r\overline{\p^2}$, $3 \le r \le 8.$} 
	\end{array} \right.
	\end{equation}
Hence this numerical invariant detects the diffeomorphism type of the smooth fiber. 

\begin{proposition}\label{lem:fiber}
Let $g \colon X' \to X$ be an orientation-preserving diffeomorphism, then the smooth fibre of $f$ is diffeomorphic to the smooth fiber of $f'$, and hence $K(X,Y)=K(X',Y')$.
\end{proposition}
\begin{proof}
As in Lemma \ref{lem:s} we write 
$$p_1(X)=s x^2 + t xy, \ \ p_1(X')=s'x'^2 + t'x'y'.$$
Then by Lemma \ref{lem:gh2}
$$p_1(X')=g^*(p_1(X))=s(\pm x'+ ny')^2 + t (\pm x' +ny') y'=sx'^2 + \cdots.$$
Therefore $s=s'$. One sees from table (\ref{eq:table}) that the smooth fiber of $f$ is diffeomorphic to the smooth fiber of $f'$, and $K(X,Y)=K(X',Y')$. 
\end{proof}

Therefore MFSs under consideration with different $K(X,Y)$ are not diffeomorphic to each other. By table (\ref{eq:table}), when $1 \le K(X,Y) \le 6$, $k=1$; when $K(X,Y)=8$, $k=2$; when $K(X,Y)=9$, $k=3$. In the following we study the diffeomorphism classification for each case. 

\subsection{The classification}
\begin{thm}\label{thm:k1}
Assume $1 \le K(X,Y) \le 6$. Then $X$ and $X'$ are oriented diffeomorphic if and only if 
$$d(X,Y) = d(X',Y'), \ \ e(X)=e(X'), \ \ K(X,Y)=K(X', Y'),$$
$$K_{X/Y}^3= K_{X'/Y'}^3 \ \ \mathrm{or} \ \  K_{X/Y}^3= -K_{X'/Y'}^3+\frac{12(d(X,Y)-1)}{d(X,Y)}K(X,Y)$$
\end{thm}
\begin{rem}
Note that when $1 \le K(X,Y) \le 5$, then condition on $d(X,Y)$ is redundant and the last equality in the condition reduces to $K_{X/Y}^3=  \pm K_{X'/Y'}^3$.
\end{rem}
\begin{proof}
In this case $k=1$, we may choose $x=c_1(X)$. Assume $p_1(X)=s x^2 + t xy$, then from the Hirzebruch-Riemann-Roch formula 
$$\langle c_1(X)^3-c_1(X)p_1(X), [X]\rangle = 48\chi( \mathcal O_X)$$
and the fact that $\chi( \mathcal O_X)=1$ we have 
\begin{equation}\label{eq:t}
t=\frac{(1-s)\langle x^3, [X]\rangle-48}{\langle x^2y, [X] \rangle}.
\end{equation}
The same holds for $X'$. 

Let $g \colon X' \to X$ be an orientation-preserving diffeomorphism. Then $e(X)=e(X')$, $K(X,Y)=K(X',Y')$ (Proposition \ref{lem:fiber}). By Proposition \ref{lem:gh2} we have $g^*(x)=\pm x' + ny'$, $g^*(y)=y'$. Therefore $\langle x^2 y, [X] \rangle = \langle x'^2 y', [X'] \rangle$. Notice that $K(X,Y)=d(X,Y) \langle x^2 y, [X] \rangle$. This shows $d(X,Y)=d(X',Y')$.

(1) If $g^*(x)=x'+ny'$, then from 
$$p_1(X')=g^*(p_1(X))=s(x'+ny')^2+t(x'+ny')y'=sx'^2+(2ns+t)x'y'$$
we have $t-t'=-2ns$. Substitute $t$ and $t'$ using equation (\ref{eq:t}) we get $(3-s)n=0$. From table (\ref{eq:table}) we have $s=  3(r-1)/(9-r) \ne 3$ for $ 3 \le r \le 8$.
Therefore $n=0$, $g^*(x)=x'$, $g^*(y')=y$. Since $K_{X/Y}=c_1(X)-c_1(Y)=x-2y$, we have 
$$K_{X/Y}^3=\langle (x-2y)^3, [X] \rangle = \langle (x'-2y')^3, [X'] \rangle =K_{X'/Y'}^3.$$ 

(2) If $g^*(x)=-x'+ny'$, then the analysis with Pontrjagin classes shows $t+t'=-2ns$. Substitute $t$ and $t'$ using equation (\ref{eq:t}) we get $(3-s)n=96/\langle x^2y, [X]\rangle$, where $\langle x^2y, [X]\rangle=(9-r)/d(X,Y)$ and $s=(3-3r)/(9-r)$. This implies $n=4d(X,Y)$. A straightforward calculation shows that $K_{X/Y}^3=- K_{X'/Y'}^3+\frac{12(d(X,Y)-1)}{d(X,Y)}K(X,Y)$.

Conversely, assume the invariants of $X$ and $X'$ coincide. If $K_{X/Y}^3=K_{X'/Y'}^3$, we define an isomorphism $\varphi \colon H^2(X,\z) \to H^2(X',\z)$ by $\varphi (x)=x'$, $\varphi(y)=y'$. We have $\langle x^2y, [X]\rangle = K(X,Y)/d(X,Y)=K(X',Y')/d(X',Y')=\langle x'^2y', [X']\rangle$. Thus the equation  $K_{X/Y}^3=K_{X'/Y'}^3$ implies $\langle x^3, [X] \rangle = \langle x'^3, [X'] \rangle$. In the expression $p_1(X) = sx^2 + txy$, the coefficients $s$ and $t$ are determined by Lemma \ref{lem:s} and equation (\ref{eq:t}), respectively. Therefore $\varphi$ preserves the first Pontrjagin class. Clearly $\varphi$ also preserves $w_2$ (since $c_1(X)=kx+ly$). Therefore there is an orientation-preserving diffeomorphism $g \colon X' \to X$. If $K_{X/Y}^3= -K_{X'/Y'}^3+\frac{12(d(X,Y)-1)}{d(X,Y)}K(X,Y)$, then we define an isomorphism $\varphi \colon H^2(X,\z) \to H^2(X',\z)$ by $\varphi (x)=-x'+4d(X,Y)y'$, $\varphi(y)=y'$. A similar calculation shows that $\varphi$ preserves the cubic form and characteristic classes.

\end{proof}

\begin{thm}\label{thm:k2}
Assume $K(X,Y)=K(X',Y')=8$. Then  $X$ and $X'$ are oriented diffeomorphic if and only  $K_{X/Y}^3=K_{X'/Y'}^3$. 
\end{thm}
\begin{proof}
In this case by Lemma \ref{lem:fiberk} and Theorem \ref{thm:mori} (2), $X$ and $X'$ are quadric hypersurfaces in projective bundles. We first analyze the topology of such bundles in detail.

Let $\xi$ be a rank $4$ holomorphic vector bundle over $Y$, with $c_1(\xi)=cy $, $y \in H^2(\p^1,\z)$ the fundamental class. Let $\pi \colon \p(\xi) \to Y$ be the projective bundle, $L^* \to \p(\xi)$ be the dual of the tautological line bundle, $u=c_1(L^*)$. Then by Leray-Hirsch Theorem $H^*(\p(\xi),\z)$ is generated as an $H^*(\p^1,\z)$-module by $\{1,u,u^2, u^3\}$, with the relation $u^4+cu^3y=0$. 

Let $j \colon X \to \p(\xi)$ be the inclusion map, $f = \pi \circ j \colon X \to Y$ a quadric hypersurface bundle with fiber $F=f^{-1}(y)$ a quadric surface in $\pi^{-1}(y)=\p^3$. The Poincar\' e dual of $j_*[X]$ equals $2u$. Let $x=j^*(u)$, then as we have seen in the proof of Lemma \ref{lem:fiberk},  $\{x, y\}$ is a basis of $H^2(X)$. The cubic form on $H^2(X)$ is given by
\begin{equation}\label{eq:cup}
\langle x^2y, [X] \rangle = \langle u^2y, j_*[X] \rangle = 2, \ \ \langle x^3, [X] \rangle =\langle u^3, j_*[X] \rangle =-2c, \ \ xy^2=y^3=0.
\end{equation} 
Let $\nu X$ be the normal bundle of $X$ in $\p(\xi)$, then $c_1(\nu X) = j^*(2u) = 2x$, and (see \S \ref{sec:projbdl})
$$TX \oplus \nu X=i^*T\p(\xi), \ \ T\p(\xi) \oplus \underline{\mathbb C} = L^* \otimes \pi^*\xi \oplus \pi^*T\p^1.$$ 
From these we compute the Chern classes of $X$:
$$c_1(X)=2x+(c+2)y, \ \ c_2(X)=2x^2+(c+4)xy, \ \ c_3(X)=(c+4)x^2y.$$
The first Pontrjagin class 
\begin{equation}\label{eq:p11}
p_1(X)=c_1(X)^2 - 2c_2(X)=2cxy,
\end{equation}
and 
\begin{equation}\label{eq:kxy}
K_{X/Y}^3 = \langle (c_1(X)-c_1(Y))^3, [X] \rangle = \langle (2x+cy)^3, [X] \rangle =8c.\end{equation}
By comparing the Euler characteristic of $X$ 
$$6-b_3(X)=e(X)=\langle c_3(X), [X] \rangle =2c+8$$
we have $b_3(X)=-2c-2$. Therefore $c \le -1$. 

Now if $g \colon X' \to X$ is an orientation-preserving diffeomorphism, then by Proposition \ref{lem:gh2},
$$g^*(p_1(X))=g^*(2cxy)= 2c(\pm x'+ny')y'=\pm 2c x'y'=p_1(X') = 2c'x'y'.$$
Therefore $c = \pm c'$. But since both $c$ and $c'$ are $\le -1$, we have $c=c'$. 
By equation (\ref{eq:kxy}) $K_{X/Y}^3=K_{X'/Y'}^3$.

Conversely, if $K_{X/Y}^3=K_{X'/Y'}^3$, then $c=c'$. We define an isomorphism $\varphi \colon H^2(X,\z) \to H^2(X',\z)$ by $\varphi(x)=x'$, $\varphi(y)=y'$. Then by equations (\ref{eq:cup}) and (\ref{eq:p11}), $\varphi$ is an isomorphism between the cubic forms and preserves $p_1$ and $w_2$. Therefore there is an orientation-preserving diffeomorphism $g \colon X' \to X$.
\end{proof}

Recall that for a prime number $p \ge 3$, the Steenrod reduced power $\mathcal P^i \colon H^n(X,\z/p) \to H^{n+2i(p-1)}(X,\z/p)$ is a stable cohomology operation such that for any $x \in H^{2i}(X,\z/p)$, $\mathcal P^i(x)=x^p$. See \cite[\S 4.L]{Hat02}.

\begin{thm}\label{thm:k3}
A Mori fiber space $X$ with $K(X,Y)=9$ is oriented diffeomorphic to $\p^1 \times \p^2$ or $\p(\mathcal O \oplus \mathcal O(-1))$. These two manifolds are distinguished by the Steenrod reduced power $\mathcal P^1 \colon H^2(X,\mathbb Z/3) \to H^6(X,\mathbb Z/3)$. For $\p^1 \times \p^2$ it is trivial whereas $\p(\mathcal O \oplus \mathcal O(-1))$ has non-trivial $\mathcal P^1$. 
\end{thm}

\begin{proof}
By Lemma \ref{lem:fiberk} the smooth fiber is $\p^2$, and by Theorem \ref{thm:mori}  $X$ is a projective bundle $\p(\xi)$ over $\p^1$, with structure group $PGL_3(\mathbb C)$. From the exact sequence 
$$\pi_1(\mathbb C^*) \to \pi_1GL_3(\mathbb C) \to \pi_1(PGL_3(\mathbb C)) \to 0$$
associated to the fiber bundle $\mathbb C^* \to GL_3(\mathbb C) \to PGL_3(\mathbb C)$ one sees that $\pi_1PGL_3(\mathbb C) \cong \z/3$. Therefore projective bundles $\p(\xi)$ over $\p^1$ are detected by $c_1(\xi) \pmod 3$. Therefore there are at most two diffeomorphism types, namely, $\p^1 \times \p^2$ and $\p(\mathcal O \oplus \mathcal O(-1))$. In the following we show that these two manifolds are not homotopy equivalent. 

Let $\{x,y\}$ and $\{x',y'\}$ be bases of $H^2(\p^1 \times \p^2)$ and $H^2(\p(\mathcal O \oplus \mathcal O(-1)))$, respectively, given before Lemma \ref{lem:fiberk}. Then 
\begin{equation}\label{eq:k=1}
x^3=0, \ x^2y=1, y^2=0; \ x'^3=1, \ x'^2y'=1, \ y'^2=0.
\end{equation} 

The Steenrod reduce power $\mathcal P^1$ is the operation
$$ \mathcal P^1\colon H^2(-,\z/3) \to H^6(-,\z/3), \ \mathcal P^1(\alpha)=\alpha^3.$$
From equation (\ref{eq:k=1}) one sees that $\mathcal P^1$ is trivial for $\p^1 \times \p^2$ while non-trivial for $\p(\mathcal O \oplus \mathcal O(-1))$. Hence these two spaces are not homotopy equivalent.

\end{proof}

\subsection{diffeomorphic MFSs with different $\dim Y$}\label{subsec:mixed}
Finally, we finish the proof of the main theorem by studying MFSs $f \colon X \to Y$ and $f' \colon X' \to Y'$ with $X$ diffeomorphic to $X'$ but $\dim Y \ne \dim Y'$.  An obvious example is $\p^1 \times \p^2$, which fibers both over $\p^1$ and $\p^2$. Another example, denoted by $\mathbb X$, is the space No.~2 in \cite[Table 12.3]{PS99}. It is a smooth Fano threefold which is a double cover of $\p^1 \times \p^2$ branched along a divisor of bidegree $(2,4)$. By \cite[Theorem 8.1.7(4)]{PS99}, $\mathbb X$ is a standard conic bundle over $\p^2$. On the other hand, it can be shown that $\mathbb X$ is a Mori fiber space over $\p^1$ with fiber a del Pezzo surface. Let $\pi \colon \mathbb X \to \p^1 \times \p^2$ be the branched double cover, $p_1 \colon \p^1 \times \p^2$ be the projection, $g= p_1 \circ \pi \colon \mathbb X \to \p^1$ be the composition. Let $z \in \p^1$ be a regular value of $g$, since $\mathbb X$ is Fano, so is $F=g^{-1}(z)$.  Let $D \subset \p^1 \times \p^2$ be the branch locus of $\pi$, then $\pi |_F \colon F \to \{z \} \times \p^2$ is a branched cover with branch locus $(\{z\} \times \p^2) \cap D$, which is a quartic curve. Therefore the equation in \cite[p.237]{BHPV} implies $p_a(F)=0$, $c_1(F)^2=2$. This shows that $F$ is del Pezzo surface diffeomorphic to $\p^2 \# 7 \overline{\p^2}$.

\begin{thm}\label{thm:mixed}
Let $f \colon X \to Y$ be a MFS with $\dim Y=2$. If $X$ is diffeomorphic to another Mori fiber space $X'$, $f' \colon X' \to Y'$ with $\dim Y'=1$. Then $X$ is oriented diffeomorphic to $\p^2 \times \p^1$ or $\mathbb X$.
\end{thm}
\begin{proof}
If $X$ is diffeomorphic to $X'$, then $b_2(X)=b_2(X')=2$, hence $b_2(Y)=b_2(X)-1=1$,
which implies $Y=\p^2$ (\cite[p.375]{BHPV}). 

Following the discussion in \S \ref{sec:2}, we choose a basis of $H^2(X,\Z)$ as follows: let $y=c_1(\mathcal O(-1))$ be a generator of $H^2(Y,\z)$, with $c_1(Y)=3y$; if $C_f = \emptyset$,  let $x=c_1(L)$; if $C_f \ne \emptyset$, let $x=c_1(X)-c_1(Y)$. Then $\{x, y\}$ is a basis of $H^2(X)$, and there are relations given in formula (\ref{eq:leray1}) or Proposition \ref{prop:h2}
\begin{equation}\label{eq:smooth}
x^2-c_1xy+c_2y^2=0,
\end{equation} 
where $c_1$, $c_2$ are integers if $C_f = \emptyset$, or half integers if $C_f \ne \emptyset$.  Let $\{x',y'\}$ be the basis of $H^2(X',\z)$ introduced before Lemma \ref{lem:fiberk}. 

If $C_f = \emptyset$, then $H^4(X,\z)$ is generated by the cup products of classes in $H^2(X,\z)$. If $X$ is diffeomorphic to $X'$, then the same holds for $X'$.
Therefore $\{x'^2,x'y'\}$ is a basis of $H^4(X',\z)$ (since $y'^2=0$). Poincar\' e duality implies $\langle x'^2y', [X'] \rangle = \pm 1$.  Then by Lemma \ref{lem:k} we have $k=1$ or $3$. 

If $C_f \ne \emptyset$, then the cup products of classes in $H^2(X,\z)$ generates $2 H^4(X,\z)$ (Proposition \ref{prop:h2}). If $X$ is diffeomorphic to $X'$, then the same holds for $X'$. In this case Poincar\' e duality implies $\langle x'^2y', [X'] \rangle = \pm 2$.  By Lemma \ref{lem:k} we have $k=1$ or $2$.

Let $\varphi \colon X \to X'$ be an orientation-preserving diffeomorphism.  Let $\varphi^*(y')=ax+ by$, with $a$, $b$ coprime integers.  Since $y'^2=0$, we have 
$$0=\varphi^*(y'^2) = (ax+by)^2=a^2x^2 + 2abxy + b^2y^2.$$
Notice that $y^2 \ne 0$, hence $a \ne 0$. By comparing with 
 equation (\ref{eq:smooth})
we have 
$$(\frac{b}{a})^2=c_2, \ \ \frac{2b}{a}=-c_1.$$
Sine $c_1$ and $c_2$ are integers or half integers, we must have  
\begin{equation}\label{eq:b}
a= \pm 1, \  c_1=\mp 2b,  \ c_2=b^2,\ \mathrm{and} \ c_1^2-4c_2=0
\end{equation}  

If $a=1$, let  $\widetilde x=x+by$; if $a=-1$, let $\widetilde x=x-by$. Then $\widetilde x$ and $y$ form a basis of $H^2(X,\Z)$, and 
\begin{equation}\label{eq:phi}
\varphi^{-1*}(\widetilde x)= \pm y', \ \ \varphi^{-1*}(y)=\pm x'+qy'
\end{equation}
for some integer $q$. We have
\begin{equation}\label{eq:28}
\widetilde x^2=\varphi^*(y'^2)=0, \ \ y^3=0, \ \  \widetilde x y^2 =  \pm y' (\pm x'+qy')^2 = \pm   x'^2 y'   = \pm 1.
\end{equation}
Therefor the cohomology ring in even degrees $H^{2*}(X,\z)$ is isomorphic to $H^{2*}(\p^2 \times \p^1, \Z)$. So is $H^{2*}(X',\Z)$.

Recall that if $C_f = \emptyset$, then $k=1$, $3$, and $H^3(X,\z)=0$. If $k=3$, then by the fact that $H^*(X',\z)$ is isomorphic to  $H^*(\p^2 \times \p^1,\z)$, Theorem \ref{thm:k3} shows that $X'$ is oriented diffeomorphic to $ \p^2 \times \p^1$. To exclude the case $k=1$, we compare the first Pontrjagin class of $X$ and $X'$. By equation (\ref{eq:class}), 
$$p_1(X)=  (c_1^2-4c_2)+ \langle c_1(Y)^2-2c_2(Y), [Y] \rangle y^2=3y^2.$$ 
Therefore 
$$p_1(X')=\varphi^{-1*}(p_1(X))=3(\pm x'+qy')^2=3x'^2 \pm 6q x'y'.$$ 
But we see from Lemma \ref{lem:s} and  table (\ref{eq:table}) that for $k=1$, in the expression of $p_1(X')$, the coefficient of $x'^2$ equals to $3(1-r)/(9-r) \ne 3$ for $3 \le r \le 8$.  This shows $k \ne 1$. 

If $C_f \ne \emptyset$, since we have shown that $c_1$ and $c_2$ are integers, Proposition \ref{prop:h2} implies that the cup products of elements in $H^2(X,\z)$ generate $2H^4(X,\z)$. Therefore the same holds for $X'$, and  Poincar\' e duality implies $\langle x'^2y', [X]\rangle = 2$. By Lemma \ref{lem:k}, $k$ equals $2$ or $1$, with the smooth fibre $F$ diffeomorphic to $\p^1\times \p^1$ or $\p^2 \#r\overline{\p^2}$ ($r=5$ or $7$), respectively. 
By Proposition \ref{prop:top2} (5) and equation (\ref{eq:smooth}), we may write $$p_1(X)=3c_1xy + py^2, \ p \in \z.$$ 
By the Hirzebruch-Riemann-Roch formula, 
$$\langle c_1(X)^3-c_1(X)p_1(X), [X] \rangle =48\chi(\mathcal O_X) =48\chi(\mathcal O_Y)=48.$$ 
From $\langle xy^2, [X] \rangle =2$ and equation (\ref{eq:b}) we get
\begin{equation}\label{eq:p}
p=3-c_2-2c_1^2=3-9b^2.
\end{equation}
Recall that $\widetilde x = x \pm b y$ and $\varphi^{-1*}(\widetilde x)= \pm y'$. Now
\begin{eqnarray*}
\varphi^{-1*}(p_1(X)) & = &\varphi^{-1*}(3c_1xy+py^2) \\
  & = & \varphi^{-1*}(3c_1\widetilde x y+(p \mp 3bc_1)y^2)  \\
  & = & (p \mp3bc_1)(\pm x'+qy')^2 + 3c_1y'(\pm x'+qy') \\
  & = & (p \mp 3bc_1)x'^2 + \cdots  
\end{eqnarray*}
Since $c_1= \mp 2b$, from equation (\ref{eq:p}) we have $p \mp 3bc_1=3-3b^2$.
Comparing this with the expression $p_1(X') = sx'^2+ tx'y'$, where value of $s$ is given in table (\ref{eq:table}), one sees that there are three possibilities, corresponding to the smooth fiber diffeomorphic to $\p^1\times \p^1$, $\p^2 \# 7\overline{\p^2}$ or $\p^2 \# 5\overline{\p^2}$.
\begin{enumerate}
\item $s=0$, $b = \pm 1$, $k=2$. Then from the fact $D_Y(2c_1)=-[C_f]$ (Proposition \ref{prop:top2} (3), note here $c_1$ equals twice of the $c_1$ in Proposition \ref{prop:h2}) we see $c_1=-2$, $c_2=1$. Then a direct calculation shows 
$$ \langle c_1(X)^3, [X] \rangle + 3 e(X) -72 \chi( \oo_X) + \frac{5}{2} [C_f] \cdot [C_f]=24+6-72+40=-2.$$
This is inconsistent with Proposition \ref{prop:top2} (4).

\item $s=-9$, $b=\pm 2$, $k=1$, the smooth fibre is $\p^2\# 7 \overline{\p^2}$.  Then $c_1=-4$, $c_2=4$, $p=-33$. By Proposition \ref{prop:top2} (3), $D_Y([C_f])=8y$, then by Proposition \ref{prop:top2} (2), $b_3(X)=64-24=40$. Note that the cohomology ring $H^*(X)$, and the characteristic classes $w_2(X)$ and $p_1(X)$ are determined by these data. Hence there exist at most one diffeomorphism type. By the discussion before this theorem, an MFS in this diffeomorphism type is provided by $\mathbb X$.

\item $s=-3$, $b^2 =2$, which is not possible.
\end{enumerate}

\end{proof}

\section{Appendix: some auxiliary lemmas}\label{sec:a}

\begin{lem}\label{lem:poincare}
Let $F$ be a smooth fibre of $f \colon X \to Y$ and $\sigma_Y \in H^4(Y,\Z)$ be the fundamental class of $Y$. Then $f^*(\sigma_Y)$ equals to the Poincar\' e dual of $[F]$ in $X$.  
\end{lem}
\begin{proof}
Let $F=f^{-1}(y)$, $D \subset Y$ be a closed $4$-disk centered at $y$. Then $N =f^{-1}(D)$ is a tubular neighborhood of $F$ in $X$. We have the following commutative diagram
$$\xymatrix{
H^4(X,\z) & H^4(X, X-\mathrm{int} N,\z) \ar[r]^{\ \ \ \ h}_{\ \ \ \ \cong} \ar[l]_{j^* \ \ \ \ }  & H^4(N, \partial N,\z) \\
H^4(Y,\z) \ar[u]^{f^*} & H^4(Y, Y- \mathrm{int} D,\z) \ar[u]^{f^*} \ar[l]_{j^* \ \ \ \ }^{\cong \ \ \ \ \ \  } & }
$$ 
where $h$ is the excision isomorphism. Now on the one hand, $hf^*(j^*)^{-1}(\sigma_Y)$ is the Thom class $U(\nu)$ of the normal bundle of $F$, on the other hand, $j^*h^{-1}(U(\nu))$ is the Poincar\' e dual of $[F]$. This finishes the proof of the lemma.
\end{proof}

We recall some basic notions of unimodular quadratic forms over the integers. Let $V$ be a unimodular quadratic form over $\z$. The \emph{norm} of a vector $x \in V$ is $x^2= x \cdot x$. The \emph{type} of $V$ is \emph{even} if $x^2$ is even for all $x \in V$, otherwise the type of $V$ is \emph{odd}. If $V$ is of even type, the type of a vector $x$ is \emph{characteristic} if $x \cdot y \equiv y^2 \pmod 2$ for all $y \in V$, otherwise the type of  $x$ is \emph{ordinary}. Wall showed in \cite{Wall62} that for a unimodular quadratic form which is neither definite nor nearly definite, the action of the orthogonal group on primitive vectors of given norm and type is transitive. 

\begin{thm}\label{thm:quadr}
Let $V$ be a unimodular quadratic form with $\mathrm{rank} V - |\mathrm{sign}(V)| \ge 4$, or $V = (1) \oplus q (-1)$ with $q \le 7$. If $v$, $v' \in V$ are two primitive vectors of the same norm, and of the same type if $V$ is of type odd, then there exists an isometry $f \colon V \to V$ such that $f(v)=v'$.
\end{thm}

This theorem is essentially used in the proof of Theorem \ref{thm:dim2singular}. In the proof of Theorem \ref{thm:dim2smooth} we need a mod $2$ version of Wall's theorem. 

\begin{lem}\label{lem:quadratic}
Let $Y$ be a compact complex surface, $V$ be the intersection form of $Y$. If $v$, $v' \in V$ are two primitive vectors such that $v^2 \equiv v'^2 \pmod 4$, and of the same type if $V$ is of type odd. Then there exists an isometry $f \colon V \to V$ such $f (v) -v' \in 2V$.  
\end{lem}
\begin{proof}
Let $V_2 = V \otimes \z/2$ be the mod $2$ reduction of $V$. For any $x \in V_2$, let $v$ be a lift of $x$ in $V$, then $\alpha (x) = v^2 \pmod 4$ is an invariant of $x$, and $\alpha(x)= 0$ or $2$ if $V$ is of even type. If $V$ is of odd type, we say the type of $x$ is characteristic or ordinary if  $v$  is characteristic or ordinary. This is independent of the choice of lifts. Then it suffices to show that for any non-zero vectors $x$, $x'$ in $V_2$ with $\alpha(x)=\alpha(x')$, and of the same type if $V$ is odd, there is an isometry $f \colon V \to V$ whose mod $2$ reduction maps $x$ to $x'$. 

If $V$ if of odd type, then by the classification of unimodular quadratic forms (\cite{MH73}), $V$ is isometric to an orthogonal sum of $(1)$'s and $(-1)$'s, i.e., we may assume $V=p(-1) \oplus q(-1)$, with $p$, $q \ge 1$ (The only simply connected compact complex surface with definite intersection form is $\p^2$). The characteristic vector of $V_2$ is $(1, \cdots , 1)$, which is fixed by any isometry of $V$. In the following we consider the case where $x$ are $x'$ are ordinary.

We first assume $p$, $q \ge 2$. In this case $\mathrm{rank}V - |\mathrm{sign}(V)| \ge 4$. Let $x$, $x'$ be ordinary with $\alpha(x)=\alpha(x')=1$. Let $v$ be a primitive lift of $x$ with $v^2 =4k+1$, let $v_0 = (2k+1, 0 \cdots, 0; 2k, 0 \cdots, 0)$. Then by the above theorem, there is an isometry $f \colon V \to V$ with $f(v)=v_0$. The mod $2$ reduction of $f$ maps $x$ to $(1, 0 \cdots, 0)$. The same applies to $x'$. For other values of $\alpha(x)$ the proof is similar. 

The remaining case is $V=(1) \oplus q(-1)$. We shall make use of some explicit isometries in the sequel. Let $u \in V$ be a vector with $u^2=-2$, then 
$$\sigma_u \colon V \to V, \ \ \sigma_u(v)=v+(v\cdot u)u$$ 
is an isometry. 

Assume $q \ge 4$.
\begin{enumerate}
\item For $x=(1;1, 0, \cdots, 0 )$, let $u=(1; 0,1,1,1,0,\cdots , 0)$, then the mod $2$ reduction of the isometry $\sigma_u$ maps $x$ to $(0; 1,1,1,1,0,\cdots, 0)$.

\item For $x=(1;1, \cdots, 1, 0, \cdots 0)$, where the number of $1$'s in the negative definite part is $k \ge 2$, and there is at least one (say, the last one) coordinate is $0$ since $x$ is not characteristic, let $u=(1; 1,1,0, \cdots, 0,1)$, then the mod $2$ reduction of $\sigma_u$, composed with a permutation of the coordinates of the negative definite part, maps $x$ to $(0;1,\cdots , 1, 0, \cdots , 0)$, where the the number of $1$'s is $k-2$.

\item For $x=(0;1, \cdots ,1, 0, \cdots, 0)$ with the number of $1$'s equals to $l \ge 5$, note that $\alpha(x) \equiv -l \pmod 4$.  Let $u=(2;1,\cdots, 1,0, \cdots, 0,1)$, where the first five coordinates in the negative definite part are $1$, then the mod $2$ reduction of $\sigma_u$, composed with a permutation of the coordinates of the negative definite part, maps $x$ to $(0;1, \cdots, 1, 0, \cdots, 0)$ with the number of $1$'s equal to $l-4$.
\end{enumerate}

Combining the above analysis, it is shown that the action of the orthogonal group of $V$ on ordinary vectors in $V_2$ with the same value of $\alpha$ is transitive. 

If $q=3$, then the vector $x=(1;1,0,0)$ has $\alpha(x)=0$; the vector $x=(1;1,1,0)$ is mapped as in (2) by an isometry to $(0;1,0,0)$ with $\alpha=-1$. The vector $x=(0;1,1,0)$ has $\alpha=-2$ and the vector $(0;1,1,1)$ has $\alpha=-3$. These are all the possibilities. The case $q=1$ or $q=2$ is trivial.

If $V$ is of even type, then it is isometric to an orthogonal sum of the hyperbolic forms $H$ and the $E_8$ forms, $V = rH \oplus sE_8$. The $11/8$-conjecture is known to be true for complex surfaces (\cite[p.378]{BHPV}), which implies that if $r < 2$ then $V=H$. The proof of this case is trivial. Now assume $r \ge 2$, so that $\mathrm{rank}V - |\mathrm{sign}(V)| \ge 4$. If $\alpha(x)=0$, let $v$ be a lift of $x$ with $v^2=4k$, $v_0=(1, 2k, 0 \cdots, 0)$. Then by the above theorem, there is an isometry $f \colon V \to V$ with $f(v)=v_0$, whose mod $2$ reduction maps $x$ to $(1, 0 \cdots, 0)$. The proof for $\alpha(x)=2$ is similar.   
\end{proof}

The following lemma is the Miyaoka-Yau inequality for simply-connected surfaces in general. Though it is a standard fact, we could not find a definitive citation.  We provide a proof here.
  
\begin{lem} \label{lem:my} For a simply connected algebraic surface $S$, the Miyaoka-Yau inequality $3c_2(S)\geq c_1^2(S)$  holds true, with equality if and only if $S$ is isomorphic to the complex projective plane
 $\p^2$.  
\end{lem}

\begin{proof}  It suffices to prove the inequality for a minimal surface. For a simply connected minimal surface $S$, the case of Kodaira dimension $\kappa (S)  \ge 0$ is already known to satisfy $c_1^2(S) < 3c_2(S)$. 
According to the Enriques-Kodaira classification of algebraic surface \cite[P.188, Table 10]{BHPV}, a simply connected minimal algebraic surface $S$ with $k(S)=-\infty$ is isomorphic to either $\p^2$ or $\p^1 \times \p^1$. A direct verification shows that all surfaces appearing in this list satisfy $c_1^2(S) \le 3c_2(S)$. Moreover, equality $c_1^2(S)=3c_2(S)$ holds if and only  if  $S \cong \p^2$. 

Now suppose $S$ is a (not necessarily minimal) surface for which we want to verify the inequality. Let $\widetilde S$ be its minimal model. If $\widetilde S$ satisfies $c_1^2(\widetilde S)<3c_2(\widetilde S)$, then applying the blow-up formula $c_2(S)= c_2(\widetilde S)+m$ and $c_1^2(S) = c_1^2(\widetilde S) -m$ for some $m \ge 0$ \cite[Proposition II.3]{Beau96} immediately yields $c_1^2(S) \le 3c_2(S)$ whenever $m \ge 1$. Thus, the inequality for all surfaces follows form the inequality for minimal surfaces.
\end{proof}

\begin{proof}[Proof of Theorem \ref{thm:dim2singular} for the case $b_2(Y) =1$.]
We give the details of the calculation showing $c_2=c_2'$.  By Proposition \ref{prop:top2} (5), the first Pontrjagin class $p_1(X)=py^2 + 3u^2$, $p \in \z$. Let $\Psi \colon H^2(X,\mathbb Q) \to \mathbb Q$ be the homomorphism, $\Psi(a)= \langle a \cup p_1(X), [X] \rangle$. Since $\Psi(y) = \langle y \cup p_1(X), [X] \rangle = 3c_1 \ne 0$,  we may take $v = u + \beta y$, $\beta \in \mathbb Q$, as a basis of $\ker(\Psi)$. Then $v$ and $y$ form a basis of $H^2(X,\mathbb Q)$, under which the cubic form on $H^2(X,\mathbb Q)$ is given by 
$$v^3 = \frac{1}{2}c_1^2+c_2+ 3 \beta c_1 + 6\beta^2, \ v^2y=c_1 + 4 \beta, \ vy^2=2, \ y^3 =0.$$
Let $g \colon X' \to X$ be an orientation preserving diffeomorphism. Since $\Psi(y)=3c_1=3c_1'=\Psi'(y')$, the induced isomorphism $g^* \colon H^2(X,\mathbb Q) \to H^2(X',\mathbb Q)$ satisfies 
\begin{equation}\label{eq:g}
g^*(v) = \lambda v', \ \ g^*(y)= y' + av', \ \lambda,  a \in \mathbb Q.
\end{equation}
Note that the transformation matrices between the bases $\{u,y\}$ and $\{v,y\}$, $\{u', y'\}$ and $\{v', y'\}$ are unimodular, and the matrix representing $g^* \colon H^2(X,\z) \to H^2(X',\z)$ under the bases $\{u,y\}$ and $\{u', y'\}$ is unimodular, therefore the matrix representing (\ref{eq:g}) is unimodular, which means $\lambda = \pm 1$. Moreover, we may assume $a \ne 0$, otherwise we are in the situation where $g^*(H^2(Y,\z) )= H^2(Y',\z)$, which has been handled.

Assume $\lambda =1$. Using the fact that $g^*$ is an isomorphism between the cubic forms, we have 
$$2=vy^2=v'(y'+av')^2=2 + 2av'^2y'+a^2v'^3,$$
from this we get $av'^3=-2v'^2y$,
\begin{equation}\label{eq:v2y}
c_1 + 4 \beta = v^2y=v'^2 (y'+av') = -v'^2y'=-(c_1+4\beta'),
\end{equation}
\begin{equation}\label{eq:v3}
\frac{1}{2}c_1^2+c_2+ 3 \beta c_1 + 6\beta^2 = v^3 = v'^3 = \frac{1}{2}c_1^2+c_2'+ 3 \beta' c_1 + 6\beta'^2.
\end{equation}
From (\ref{eq:v2y}) one has $3 \beta c_1 + 6\beta^2=3 \beta' c_1 + 6\beta'^2$. Hence by (\ref{eq:v3}) $c_2=c_2'$. The proof of the case $\lambda =-1$ is similar. 
\end{proof}

\begin{proof}[Proof of Theorem \ref{thm:dim2singular} for the case $b_2(Y) =2$.]
By choosing the symplectic bases appropriately, we may assume $g^*(e_1)=e_1'$. 
Write 
$$g^*(e_2)=ax'+be_1' +ce_2', \ a, b, c \in \mathbb Q$$
then $(ax'+be_1' +ce_2')^2=0$ (since $e_2^2=0$). By taking product with $e_1'$, $e_2'$ and $x'$, respectively, we have 
$$ac=0, \ ab=0, \ a^2x'^3-16bc=0.$$
Note that if $a=0$, then $g^*(H^2(Y, \z))=H^2(Y',\z)$, we reduce to the case which has been handled.  Now assume $a \ne 0$. Then $b=c=0$ and $x'^3=0$ (which implies $8c_2'-c_1'^2=0$ by (\ref{eq:cubic_sing})). By the symmetry of $X$ and $X'$, we also have $x^3=0$. Therefore
$$g^*(e_1)=e_1', \ g^*(e_2)=ax', \ g^*(x)=\alpha x'+\beta e_1' + \gamma e_2'.$$ 
The identities $x^2e_1=0$ and $x^2e_2=0$ imply 
$$(\alpha x'+\beta e_1' + \gamma e_2')^2 e_1'=0, (\alpha x'+\beta e_1' + \gamma e_2')^2 x'=0,$$
which further imply $\alpha \gamma=0$ and $\beta \gamma=0$. Note that $\gamma \ne 0$ since $g^*$ is an isomorphism. Therefore $\alpha=\beta=0$. From (see (\ref{eq:cubic_sing}))
$$-8 = x e_1 e_2 = g^*(x e_1 e_2) = a \gamma x' e_1' e_2' =-8 a \gamma $$
we get $a \gamma =1$. To simplify the notations we rename the coefficients and write
$$g^*(e_1)=e_1', \ g^*(e_2) = \alpha x', \ g^*(x)=\beta e_2'$$
with $\alpha \beta =1$.

Denote $c_1=a e_1 + b e_2$, $c_1'=a'e_1'+b'e_2'$. Then by Proposition \ref{prop:top2} (5) 
$$p_1(X) e_1 = 3u^2 e_1 = \frac{3}{2}(c_1u-c_2)e_1=\frac{3}{2}bue_1e_2, \ p_1(X)e_1'=\frac{3}{2}b'u'e_1'e_2'.$$
Since $g^*(p_1(X)e_1)=p_1(X') e_1'$ and $ue_1e_2=u'e_1e_2' \ne 0$, we have $b=b'$. The image of $u$ is  
$$g^*(u)=g^*(\frac{1}{4}(x+c_1))=\alpha bu'+\frac{1}{4}(a-a'b\alpha)e_1'+\frac{1}{4}(\beta-bb'\alpha)e_2'$$
Then we get (using the fact that $\alpha \beta =1$)
$$g^*(e_1e_2u)=e_1'\alpha x'g^*(u)=(1-2\alpha^2 bb')e_1'e_2'u'.$$
This implies $b=b'=0$. 
\end{proof}

\end{document}